\documentclass[11pt]{amsart}
\usepackage{amssymb}
\usepackage{amsmath}
\usepackage{amsthm}
\usepackage{amscd}
\usepackage{amsfonts}
\usepackage{tikz}
\usepackage{tikz-qtree}
\usepackage{tikz-cd}
\usepackage{mathrsfs}
\usetikzlibrary{trees}
\usetikzlibrary{calc}
\usepackage[all]{xy}
\usepackage{faktor}
\usepackage{graphicx}
\usepackage{graphics}

\newtheorem{theorem}{Theorem}
\newtheorem{lemma}{Lemma}
\newtheorem{proposition}{Proposition}

\newtheorem{conjecture}{Conjecture}
\theoremstyle{definition}%%Change Theoremstyle

\newcommand{\RN}[1]{%for Roman numerals
  \textup{\uppercase\expandafter{\romannumeral#1}}%
}

\calclayout

\newcommand{\barc}{\operatorname{B}}

\newcommand{\dg}{\operatorname{dg}}
\newcommand{\rank}{\operatorname{rank}}
\newcommand{\End}{\operatorname{End}}
\newcommand{\coEnd}{\operatorname{coEnd}}
\newcommand{\PBW}{\operatorname{PBW}}

\newcommand{\As}{\operatorname{As}}

\newcommand{\Id}{\operatorname{Id}}

\newcommand{\Hom}{\operatorname{Hom}}

\newcommand{\Homology}{\operatorname{H}}

\newcommand{\degree}{\operatorname{deg}}

\begin{document}

\title{Minimal Models of Some Differential Graded Modules}

\author{  BERR\.{I}N \c SENT\" URK$^*$,  \" OZG\" UN \" UNL\" U }
\address{ Department of Mathematics,
              TED University, Ankara, 06420, Turkey }
\email{ berrin@fen.bilkent.edu.tr }
\address{ Department of Mathematics, Bilkent
University, Ankara, 06800, Turkey.}

\email{ unluo@fen.bilkent.edu.tr }

\thanks{The authors are partially supported by T\"UB\.ITAK-TBAG$/117$F$085$}
\thanks{$^*$Corresponding author}

\maketitle
\begin{abstract}
Minimal models of chain complexes associated with free torus actions on spaces have been extensively studied in the literature. In this paper, we discuss these constructions using the language of operads.  The main goal of this paper is to define a new Koszul operad that has projections onto several of the operads used in these minimal model constructions.
\end{abstract}

\section{Introduction}

\label{intro}
Let $k$ be an algebraically closed field of characteristic $2$ and $G$ an elementary abelian $2$-group of rank $r$. Considering the chain complexes associated with free $G$-spaces, one obtains an algebraic conjecture stronger than the Halperin-Carlsson rank conjecture about $2$-torus actions ($G$-actions). For any chain complex $C$ of $k$-modules, we denote the homology of $C$ by $\Homology(C)$.
\begin{conjecture}\label{conjchaincomplex}
	If $C$ is a finite chain complex of free $kG$-modules with $\Homology(C)\neq 0$, then $\dim_k	\Homology(C)$ is at least $2^r$.
\end{conjecture}
Considering the polynomial ring $S:=k[x_1,\ldots,x_r]$, an equivalent algebraic conjecture is given by Proposition $\RN{2}.1$ and $\RN{2}.2$ in \cite{Carlssonbeta}.  We say $(M,\partial)$ is a \emph{differential graded $S$-module} ($\dg$-$S$-module)
if $M$ is an $S$-module, and $\partial$ is an $S$-linear endomorphism of
$M$ that has degree $-1$ and satisfies $\partial^2  = 0$.
Moreover, we say a $\dg$-$S$-module is \emph{free} if its underlying graded $S$-module is free.
\begin{conjecture}\cite[Conjecture~\RN{2}.8]{Carlsson1986}\label{carlssonsconjecture}
Let $S=k[x_1,\ldots,x_r]$ be the polynomial
algebra in $r$ variables of degree $-1$ with coefficients in an algebraically closed field of characteristic $2$. If
$(M,\partial)$ is a free, finitely generated $\dg$-$S$-module with $0<\dim_k	\Homology(M)< \infty $, then ${\rank_S M \geq 2^r}$.
\end{conjecture}

 In the literature, bounds for the dimension of $\Homology(C)$ in Conjecture \ref{conjchaincomplex} and for the rank of $M$ in Conjecture \ref{carlssonsconjecture} are obtained by studying minimal models of $C$ and $M$. In \cite{Carlsson2}, Carlsson showed the existence of minimal models of certain free differential graded $S$-modules. Here we give an explicit construction of these minimal models using operad theory. Again using operads we construct minimal models of chain complexes of Borel constructions of spaces with a free $G$-action. These minimal models are equivalent to  minimal Hirsch-Brown models given by Allday-Puppe \cite{AlldayPuppe}.

 Note that Conjecture \ref{conjchaincomplex} holds if we further assume that the Euler characteristic of $C$ is non-zero. More precisely,
$$\chi(C)=\left|G\right|  \chi(k\otimes_{kG} C)=\chi(\Homology(C))=\sum\limits_{i\geq 0}(-1)^i \dim_k{\Homology_i(C)}\neq 0.$$
Hence the dimension of total homology $\dim_k{\Homology(C)}\geq \left|G\right| = 2^r$.
Due to the equivalence of conjectures one could ask if a similar result holds for Conjecture \ref{carlssonsconjecture}. We prove the conjecture in the following case:
\begin{theorem}\label{thm1}
Conjecture \ref{carlssonsconjecture} holds if every integer $n$, $m$ have the same parity whenever
 ${\Homology_n(M)\neq 0}$ and ${\Homology_m(M)\neq 0}$. In fact, $\chi(\Homology(M)):=\sum\limits_{i\geq 0}(-1)^i \dim_k{\Homology_i(M)}\neq 0$ implies Conjecture \ref{carlssonsconjecture}.
\end{theorem}

When the characteristic of the field is odd, a result analogous to Theorem \ref{thm1} is proved by Walker \cite{Walker}, \cite{WalkerARC}.

Puppe \cite{Puppe} asserted that, given a certain multiplicative structure on the minimal Hirsch-Brown model for the equivariant cohomology
of a space with a free torus action, these bounds can be tightened to verify the Halperin-Carlsson rank conjecture. The main goal of this paper is to  put a multiplicative structure on minimal Hirsch-Brown models of $G$-spaces. Note that the group algebra $kG$ is an exterior algebra since $k$ is a characteristic $2$ field. First we consider the group algebra $kG$ and the polynomial algebra $S=k[x_1,\ldots,x_r]$ as algebraic operads where all non-trivial operations are unary operations. Then to put multiplicative structures on our minimal models we define a new Koszul operad.

\begin{theorem}\label{mainthm}Let $k$ be an algebraically closed field of characteristic $2$ and $G$ an elementary abelian $2$-group of rank $r$. Then there exists an algebraic operad $\mathscr{P}$ in the category of differential graded modules over $k$ such that $\mathscr{P}$ has the following properties:
\begin{itemize}
\item[(i)] The unary operations of $\mathscr{P}$ with the composition of $\mathscr{P}$ considered
as multiplication is isomorphic to the group algebra $kG$;
\item[(ii)] $\mathscr{P}$ has an associative binary operation $\mu$;
\item[(iii)] $\mathscr{P}$ is a Koszul operad;
\item[(iv)] The Koszul dual operad of $\mathscr{P}$ has projections onto the associative operad $\mathrm{As}$ and the polynomial algebra $S=k[x_1,\ldots,x_r]$ where we consider $S$ as an operad whose all nontrivial operations are unary;
\item[(v)] For every $G$-space $X$ the singular cochain complex $C^{\bullet}(X;k)$ has a $\mathscr{P}$-algebra structure whose restriction to the unary operations of $\mathscr{P}$ gives the natural $kG$-module structure on $C^{\bullet}(X;k)$ and the action of $\mu $ is the same as the dual of the Alexander-Withney diagonal map.
\end{itemize}
\end{theorem}

Let $\mathscr{P}$ be the operad in Theorem \ref{mainthm} and $\iota:\mathscr{P}\rightarrow \Omega\mathscr{P}$ be the universal twisting morphism. Given a space $X$ that admits a free $G$-action, we will consider the bar construction  $\barc_{\iota} \Homology(C^{\bullet}(X;k))$ as the minimal Hirsch-Brown model of $X$; see Section \ref{minmodelsection}.

Throughout this paper, $k$ is an algebraically closed field of characteristic $2$ and all (co)operads are non-symmetric (co)operads in the category of  $\dg$-modules over $k$.
In Section \ref{sec:outlineofanapplication}, we recall Puppe's method to find lower bounds on total homology dimension of complexes with a free $G$-action and give an outline of our method.
In Section \ref{sec:1}, we recall definitions, notation, and well-known results about algebraic (co)operads.
In Section \ref{sect:MinimalHirschBrown}, we discuss constructions of minimal models and prove Theorem \ref{thm1}.
In Section \ref{secMult}, we prove the our main result Theorem \ref{mainthm} and its applications.
\section{The outline of an application of Theorem \ref{mainthm}}
\label{sec:outlineofanapplication}
Assume that $r$ is a positive integer and $m$ is a nonnegative integer. Let $S$ denote the polynomial algebra $k[x_1,\ldots,x_r]$ with $\degree(x_i)=-1$ and  $\Lambda_m$ denote the exterior algebra $\Lambda(z_{1}^{(m)},\ldots, z_r^{(m)})$ with $\degree(z_i^{(m)})=-m$ for all $i$ in $\{1,\ldots, r\}$.
Note that according to our degree conventions Puppe \cite{Puppe} defines the Koszul complex $K_r(m)$ corresponding to the regular sequence $(x_1^{m+1},\ldots, x_r^{m+1})$ in $S$ as the differential graded algebra
$ K_r(m)=\left(S \tilde{\otimes} \Lambda_m, \partial\right), $
 where the differential $\partial$ determined by $\partial(x_i\otimes 1)=0$ and $\partial(1\otimes z_i^{(m)})=x_i^{m+1}$. Then Puppe considers $\dg$-$S$-module morphisms $\gamma: K_r(m)\rightarrow K_r(0)$ which lift the projection
 $$\epsilon:\Homology (K_r(m))\cong S/(x_1^{m+1},\ldots, x_r^{m+1})\longrightarrow  S/(x_1,\ldots, x_r)\cong \Homology (K_r(0)). $$
 Puppe denotes the rank of $\gamma$ by $\rank(\gamma)$ which is the rank of the localization of $F\otimes \gamma$ where $F$ is the field of fraction of $S$.
In \cite[Lemma~2.1.a]{Puppe}, Puppe shows that if $\gamma$ also preserves the multiplicative structure, then $\rank(\gamma)=2^r$. In \cite[Lemma~2.1.b]{Puppe}, Puppe asserts that $\rank(\gamma)\geq 2r$ without the assumption about the multiplicative structure.

Let $G$ be an elementary abelian $2$-group of rank $r$ generated by $g_1, g_2, \ldots, g_r$. Puppe defined a minimal Hirsch-Brown model $M=S\tilde{\otimes}\Homology^{\bullet}(X;k)$ associated to a free action of $G$ on $X$, where $X$ is a finite dimensional $CW$-complex. Then in \cite[Proposition~4.1]{Puppe}, Puppe showed that there exists a map  $\alpha: K_r(m)\rightarrow M$ that induces a surjective map on the zeroth homology of these complexes for large enough $m$. Moreover, in  \cite[Proposition~4.2]{Puppe}, Puppe showed that there exists a map $\beta: M \rightarrow K_r(0)$ that induces a surjective map on the zeroth homology. Now notice that the rank of  $\gamma=\beta\circ \alpha$ is less than or equal to the total homology dimension  $\sum_{i=0}^{\infty}\dim_k \Homology_i (X;k)$ of $X$. Hence Puppe used this idea to put lower bounds on the total cohomology dimension of complexes with a free $G$-action.

 We will use an idea similar to Puppe's idea discussed above. In our setting Koszul complexes are dual of the ones considered by Puppe. Now we give the definitions in our setting. Assume $G$ acts on the product of $r$ equidimensional spheres $\mathbb{S}^m \times \ldots \times \mathbb{S}^m$ such that $g_i$ acts on the $i$th sphere with the antipodal action. Then we will denote the minimal Hirsch-Brown model $\barc_{\iota} \Homology(C^{\bullet}(X;k))$ associated to this action by $\tilde{K_r}(m)$ and call it a Koszul complex for our operad in Theorem \ref{mainthm}. We use this terminology because when the above bar construction is done using the suboperad of $\mathscr{P}$ generated by its unary operations, we obtain the usual dual Koszul complex $S^*\tilde{\otimes} \Lambda_m$ which we still denote by $K_r(m)$ by abuse of notation. Hence from now on,
 $$ K_r(m)=\left(S^* \tilde{\otimes} \Lambda_m, \partial\right), $$
 where the differential $\partial$ determined by $\partial(x_i\otimes 1)=0$ and
 $$\partial(x_i^n\otimes z_i^{(m)})=\left\{
 \begin{array}{cc}
   0 & \text{ if } n< m+1 \\
   x_i^{n-m-1}& \text{ if } n\geq m+1.
 \end{array}
 \right.$$

For the rest of the paper $i_m$ will denote the natural inclusion of $\mathbb{S}^0 \times \ldots \times \mathbb{S}^0$ in $\mathbb{S}^m \times \ldots \times \mathbb{S}^m$ where each $\mathbb{S}^0$ is sent to the south and north poles of the corresponding $\mathbb{S}^m$. This continuous function induces coalgebra morphisms $i^*_m:\tilde{K_r}(m)\rightarrow \tilde{K_r}(0)$ and  $i^*_m:K_r(m)\rightarrow K_r(0)$. Hence, in \cite[Lemma~2.1.a]{Puppe} and \cite[Lemma~2.1.b]{Puppe}, Puppe gave lower bounds on the rank of certain maps between Koszul complexes $K_r(m)$ and $K_r(0)$ which induces the same map as $i_m$ does between the homology of these complexes.  Puppe in  \cite[Corollary~5.2]{Puppe} asserted that if the minimal Hirsch-Brown model of a finite dimensional space $X$ with a free $G$-action carries differential graded algebra structure, then  $\sum_{i=0}^{\infty}\dim_k \Homology^i (X;k)\geq 2^r$. Note that the multiplicative structure is not considered in  \cite[Proposition~4.1]{Puppe} and  \cite[Proposition~4.2]{Puppe}. However, in the proof of \cite[Corollary~5.2]{Puppe} one needs extensions of \cite[Proposition~4.1 and ~4.2]{Puppe} to differential graded algebras. More precisely, the morphism $\alpha$ must be constructed in a way compatible with the multiplicative structure. However, the almost random selections of images of $\alpha $ in the homological proofs given in the literature for these results  do not take the multiplicative structure into consideration. For example, some selections of $\alpha $ fail for the associated minimal model of the free $\mathbb{Z}/2\mathbb{Z}\times \mathbb{Z}/2\mathbb{Z} $ action on $\mathbb{R}P^3$ induced by quotienting out the center of the free action of $Q_8$ on $Sp(1)\cong \mathbb{S}^3$ by the left multiplication. Here we partially fill the gap by the following Proposition.

\begin{proposition}\label{mainprop}
Let $G$ be an elementary abelian $2$-group and act freely on a finite-dimensional simplicial set $X$. Assume $\mathscr{P}$ is the operad in Theorem \ref{mainthm}. Let $\tilde{M}$ denote the minimal Hirsch-Brown model of $X$; in other words $\tilde{M}=\barc_{\iota}(\Homology(C^{\bullet}(X;k)))$. Then there exists a positive integer $m$ and $\mathscr{P}^{\mbox{!`}}$-coalgebra morphisms $\alpha$, $\beta$ such that the composition
$$ \tilde{K_r}(m)\overset{\alpha}\rightarrow \tilde{M}\overset{\beta} \rightarrow \tilde{K_r}(0), $$
sends $K_r(m)$ to $K_r(0)$ and which induces the same map from $\Homology (K_r(m))$ to $ \Homology( K_r(0))$ as $i_m^*$ does.
\end{proposition}
However, using the above proposition we cannot confirm the main result of Puppe using \cite[Lemma~2.1.a]{Puppe} as our multiplicative structure does not have all the properties Puppe used in the proof of \cite[Lemma~2.1.a]{Puppe}. On the other hand \cite[Lemma~2.1.b]{Puppe} is proved by only using the differential graded $S$-module structure on these Koszul complexes. Hence to improve the results about bounds for the total dimension of the cohomology of a space that admits a free $G$-action, it is enough to consider maps between $\tilde{K_r}(m)$ and $\tilde{K_r}(0)$ and prove analogs of \cite[Lemma~2.1.b]{Puppe}.
The following proposition is similar to \cite[Lemma~2.1.b]{Puppe} but proved using our setting where we fix an isomorphism $\tilde{K_r}(m)\cong \mathscr{P}^{\mbox{!`}}\circ \Lambda_m$ with the composition $\circ$ defined as in \cite[Section~5.9.1]{Loday}.

\begin{proposition}\label{mainprop2}Let $\mathscr{P}$ denote the operad in Theorem \ref{mainthm}, $m$ be a positive integer and $\gamma:{K_r}(m)\rightarrow {K_r}(0)$  be a $S^*$-coalgebra morphism which induces the same map from $\Homology (K_r(m))$ to $\Homology( K_r(0))$ as $i_m^*$ does. Assume $\gamma$  extends to a $\mathscr{P}^{\mbox{!`}}$-coalgebra morphism $\tilde{\gamma}:\tilde{K_r}(m)\longrightarrow \tilde{K_r}(0)$ whose restriction to $\mathscr{P}^{\mbox{!`}}\circ 1$ is induced by the identity on $\mathscr{P}^{\mbox{!`}}$. Then the linear map $F\otimes _S \gamma^*$ has rank at least $2r$ where $F$ denotes the field of fractions of the ring $S$.
\end{proposition}

As discussed above Proposition \ref{mainprop2} can be used to give lower bounds for the total homology of the finite dimensional complexes with a free $G$-action.

\section{Definitions and Notation}
\label{sec:1}
We will take most of definitions and notation from \cite{Loday} and \cite{Hoffbeck}.
\subsection{Free (co)operads}\label{treesection}
A \emph{$\dg$-$\mathbb{N}$-module} $M$ is a sequence of differential graded $k$-modules
$$M=(M(0),M(1),M(2),\dots ).$$

A \emph{free operad} over the $\dg$-$\mathbb{N}$-module $M$ is an operad $\mathscr{T}(M)$ together with a
$\dg$-$\mathbb{N}$-module morphism $i:M \rightarrow \mathscr{T}(M)$ such that
if $\mathscr{P}$ is an operad and $f:M\rightarrow \mathscr{P}$ is an $\dg$-$\mathbb{N}$-module morphism
then there exists a unique operad morphism
$\tilde{f}:\mathscr{T}(M)\rightarrow \mathscr{P} $ with $f= \tilde{f}\circ i$.

There exists a free operad $\mathscr{T}(M)$ over every $\dg$-$\mathbb{N}$-module $M$, see \cite[Section~5.9.6]{Loday}. Let $n(v)$ denote the number of leaves of a vertex $v$ in a tree and $\tau(M)$ be the tensor product of $M(n(v))$'s as $v$ ranges over the vertices of a tree $\tau$. As a $\dg$-$\mathbb{N}$-module, $\mathscr{T}(M)$ is the direct sum of $\tau(M)$'s where $\tau$ ranges over all planar trees. The operad composition of $\mathscr{T}(M)$ is given by grafting trees.
Hence, as an operad $\mathscr{T}(M)$ is generated by $\mathcal{B}^{\mathscr{T}}(M)$, that is,

$$
\sbox0{$\begin{array}{c}
\hspace{7pt}
\begin{tikzpicture}[scale=0.7,baseline=-1mm]
\draw [thick](0,-0.1) -- (0,-0.5);
\draw  node at (0,0.25) [draw,circle, scale=0.6]{\Large{b}};
\end{tikzpicture}
\end{array}\Big\vert \, \textup{b}\in \mathcal{B}_0 $ }
\mathopen{\resizebox{1.2\width}{\ht0}{$\Big\{$}}
\usebox{0}
\mathclose{\resizebox{1.2\width}{\ht0}{$\Big\}$}}
\hspace{2pt}
\bigcup
\hspace{2pt}
\sbox0{$\begin{array}{c}
\hspace{7pt}
\begin{tikzpicture}[scale=0.7,baseline=-1mm]
\draw [thick](0,0.6) -- (0,0.9);
\draw [thick](0,-0.1) -- (0,-0.5);
\draw  node at (0,0.25) [draw,circle, scale=0.6]{\Large{b}};
\end{tikzpicture}
\end{array}\Big\vert \, \textup{b}\in \mathcal{B}_1 $ }
\mathopen{\resizebox{1.2\width}{\ht0}{$ \Big\{$ }}
\usebox{0}
\mathclose{\resizebox{1.2\width}{\ht0}{$\Big\}$ }}
\hspace{2pt}
\bigcup
\hspace{2pt}
\sbox0{$\begin{array}{c}
\hspace{7pt}
\begin{tikzpicture}[scale=0.7,baseline=-1mm]
\draw [thick](-0.25,0.5) -- (-0.55,0.9);
\draw [thick](0.25,0.5) -- (0.55,0.9);
\draw [thick](0,-0.1) -- (0,-0.5);
\draw  node at (0,0.25) [draw,circle, scale=0.6]{\Large{b}};
\end{tikzpicture}
\end{array}\Big\vert \, \textup{b}\in \mathcal{B}_2 $ }
\mathopen{\resizebox{1.2\width}{\ht0}{$\Big\{$}}
\usebox{0}
\mathclose{\resizebox{1.2\width}{\ht0}{$\Big\}$}}
\hspace{2pt}
\bigcup
\hspace{2pt}
\ldots
$$
where $\mathcal{B}_j$ is a basis for $M(j)$ as a $k$-vector space.
The $\dg$-$\mathbb{N}$-module $\mathscr{T}(M)$ is always equipped with an extra grading, called the weight-grading.
If $M$ itself has no such extra grading, then the trees in $\mathscr{T}(M)$ with exactly $n$-vertices are said to have weight-grading $n$.
If $M$ already has weight-grading, then the sum of weight-grades of elements in $M$ used to label the vertices of a tree in $\mathscr{T}(M)$
is the weight-grade of that tree.
Hence we have a decomposition of $\mathscr{T}(M)$ indexed by the weight-grading
$$\mathscr{T}(M)=\bigoplus _{n\geq 0}\mathscr{T}(M)^{(n)}  \, ,$$
where each $\mathscr{T}(M)^{(n)}$ is a $\dg$-$\mathbb{N}$-module.

Dually, we let $\mathscr{T}^c(M)$ denote the cofree cooperad over $M$. $\mathscr{T}^c(M)$ is isomorphic to $\mathscr{T}(M)$
as a weight-graded $k$-vector space, while as a cooperad $\mathscr{T}^c(M)$ is cogenerated by the generators of $\mathscr{T}(M)$ mentioned above.

Let  $sM$ denote the $\dg$-$\mathbb{N}$-module $M$ whose degree is shifted by $1$,
i.e., $sM_i(n)=M_{i-1}(n)$ for $n\in \mathbb{N}$ and $i\in \mathbb{Z}$. More generally, for any integer $m$, $s^mM$ denotes the $\dg$-$\mathbb{N}$-module $M$ whose degree is shifted by $m$.

\subsection{Quadratic (co)operads}\label{quadraticcoop}

A pair $(M,R)$ is called an \emph{operadic quadratic data pair} if $M$ is a $\dg$-$\mathbb{N}$-module and $R$ is a sub-$\dg$-$\mathbb{N}$-module of $\mathscr{T}(M)^{(2)}$. The \emph{quadratic operad} associated to the quadratic data pair $(M,R)$ is
$$\mathscr{P}(M,R):=\mathscr{T}(M)/( R ),$$
where $( R )$ is the operatic ideal generated by $R\subseteq \mathscr{T}(M)^{(2)}$. In other words,
$\mathscr{P}(M,R)$ is the largest quotient operad $\mathscr{P}$ of $\mathscr{T}(M)$ for which the composite
$$R \rightarrowtail \mathscr{T}(M)^{(2)}\rightarrowtail \mathscr{T}(M)\twoheadrightarrow \mathscr{P}$$
is zero.

Dually, the \emph{quadratic cooperad} $\mathscr{C}(M,R)$
 associated to the quadratic data pair $(M,R)$ is the largest subcooperad of $\mathscr{T}^c(M)$ for which the composite
$$\mathscr{C} \rightarrowtail \mathscr{T}^{c}(M)\twoheadrightarrow \mathscr{T}^c(M)^{(2)} \twoheadrightarrow \mathscr{T}^c(M)^{(2)}/R$$
is zero, see \cite[Section~7.1]{Loday} and \cite[Section~6.3.1]{Dotsenko}.

 The \emph{Koszul dual cooperad} of a quadratic operad $\mathscr{P}=\mathscr{P}(M,R)$ is
$$\mathscr{P}^{\mbox{!`} }:=\mathscr{C}(sM,s^2R),$$
where $s^2R$ is the image of $R$ under the natural map ${\mathscr{T}(M)^{(2)}\rightarrow\mathscr{T}(sM)^{(2)}}$.
Similarly, the \emph{Koszul dual operad} of a quadratic cooperad ${\mathscr{C}=\mathscr{C}(M,R)}$ is
$$\mathscr{C}^{\mbox{!`} }:=\mathscr{P}(s^{-1}M,s^{-2}R),$$
where $s^{-2}R$ is the image of $R$ under the map ${\mathscr{T}(M)^{(2)}\rightarrow\mathscr{T}(s^{-1}M)^{(2)}}$
 induced by the natural degree $1$ $\dg$-$\mathbb{N}$-module morphism $M$ to $sM$, see \cite[Section~7.4.7]{Loday}.

\subsection{The (co)bar construction }
For an operad $\mathscr{P}$, let $\overline{\mathscr{P}}$ be the cokernel of the unit map $I\rightarrow \mathscr{P}$.
 If $\mathscr{P}=I\oplus \overline{\mathscr{P}}$ as
$\dg$-$\mathbb{N}$-modules, the \emph{bar construction} $B\mathscr{P}$ of $\mathscr{P}$ is the $\dg$-cooperad $\mathscr{T}^c(\,s\,\overline{\mathscr{P}})$
with differential $d_1+d_2$,
where $d_1$ and $d_2$ are as in \cite[Section~6.5.1]{Loday}.

Similarly, for a cooperad $\mathscr{C}$, let $\overline{\mathscr{C}}$ denote the kernel of the
counit map $\mathscr{C}\rightarrow I$. If $\mathscr{C}=I\oplus \overline{\mathscr{C}}$
as $\dg$-$\mathbb{N}$-modules, the \emph{cobar construction} $\Omega \mathscr{C}$ of $\mathscr{C}$ is the $\dg$-operad
$\mathscr{T}(\,s^{-1}\,\overline{\mathscr{C}})$ with differential $d_1+d_2$,
where $d_1$ and $d_2$ are as in \cite[Section~$6.5.2$]{Loday}.

Let $(M,R)$ be an operatic quadratic data pair.
The quadratic operad $\mathscr{P}=\mathscr{P}(M,R)$ is \emph{Koszul} if
the natural $\dg$-cooperad morphism ${\mathscr{P}^{ \mbox{!`} } \rightarrow \barc \mathscr{P}}$
is a quasi-isomorphism of $\dg$-cooperads, see \cite[Theorem~7.4.2]{Loday}.
 When $\mathscr{P}$ is Koszul, we define the operad $\mathscr{P}_{\infty }:=\Omega \mathscr{P}^{\mbox{!`} }$.

\subsection{(Co)algebras over (co)operads}\label{coalgebras}
Let $\mathscr{P}$ be an operad. A \emph{$\mathscr{P}$-algebra} is a differential graded $k$-module $A$
 together with an operad morphism
$\mathscr{P}\rightarrow\mathop{\End}_A$, where $\mathop{\End}_A(n)=\Hom(A^{\otimes n},A)$.
Dually, for a cooperad $\mathscr{C}$,
a \emph{$\mathscr{C}$-coalgebra} is a differential graded $k$-module $C$ together with an operad morphism
${\mathscr{C}}^* \rightarrow \mathop{\coEnd}_C$, where ${\mathscr{C}}^*$ is the dual of ${\mathscr{C}}$ and
$\mathop{\coEnd}_C(n)=\Hom(C,C^{\otimes n})$. For differential graded $k$-module $C$, we define an operad morphism $$\psi_C:\mathrm{coEnd}_C\rightarrow \mathrm{End}_{C^*}$$ which sends $\alpha:C\rightarrow C^{\otimes n}$ to the composition $(C^*)^{\otimes n}\overset{i_C}{\rightarrow } (C^{\otimes n})^{*}\overset{\alpha^*}\rightarrow C^*$ where $i_C$ is defined by $$i_C(f_1\otimes \dots \otimes f_n)(v_1\otimes \dots \otimes v_n)=f_1(v_1)\dots f_n(v_n)$$ for every $f_1,\dots,f_n\in C^*$ and $v_1,\dots,v_n$ in $C$. Given a dg-$\mathscr{C}$-coalgebra $C$, we get dg-$\mathscr{C}^*$-algebra structure on $C^*$ by the composition  $\mathscr{C}\rightarrow \mathrm{coEnd}_C\overset{\psi_C}{\rightarrow }\mathrm{End}_{C^*}$.

A cooperad $\mathscr{C}$ is called \emph{coaugmented} if its counit map has a right inverse. Let $C$ be a coalgebra over coaugmented cooperad $\mathscr{C}$. For $x\in C$, we define $x_1,x_2,\ldots$ by
$$\Delta_C(x)=(x_1,x_2,\ldots)\in  \prod_{n\geq 1}\left(\mathscr{C}(n)\otimes C^{\otimes n}\right),$$ where $\Delta_C$ denotes the structure map of the coalgebra $C$.
We filter the coalgebra $C$ by $F_rC:=\{x\in C \, | \, x_i=0 \mbox{ for any } i>r \}$
for $r\geq 1$. If $C=\bigcup_{r\geq 1}F_rC$, then $C$ is conilpotent, see \cite[Section~5.8.4]{Loday}.

Let $\mathscr{C}$ be a $\dg$-cooperad, $\mathscr{P}$ a $\dg$-operad,
 and $\varphi:\mathscr{C}\rightarrow \mathscr{P}$ a twisting
morphism as in \cite[Section~11.1.1]{Loday}. The \emph{bar construction} ${\barc}_{\varphi}$ is a functor from the category of $\dg$-$\mathscr{P}$-algebras to
 the category of conilpotent $\dg$-$\mathscr{C}$-coalgebras, defined on a $\dg$-$\mathscr{P}$-algebra $A$ by
 $${\barc}_{\varphi}A:=(\mathscr{C}\circ_{\varphi}\mathscr{P})\circ_{\mathscr{P}}A,$$
 where $\circ_{\varphi}$ denotes the right-twisted composite product and
 $\circ_{\mathscr{P}}$ denotes the relative composite product over $\mathscr{P}$, see \cite[Sections~6.4.7 and 11.2.1 ]{Loday}.

Dually, the \emph{cobar construction}
$\Omega_{\varphi}$ is a functor from the category of conilpotent $\dg$-$\mathscr{C}$-coalgebras to the category of
$\dg$-$\mathscr{P}$-algebras, defined on a conilpotent $\dg$-$\mathscr{C}$-algebra $C$ by
$$\Omega_{\varphi}C:=(\mathscr{P}\circ_{\varphi}\mathscr{C})\circ^{\mathscr{C}}C,$$
where $\circ_{\varphi}$ denotes the left-twisted composite product and
 $\circ^{\mathscr{C}}$ denotes the relative composite product over $\mathscr{C}$, see \cite[Sections~6.4.7 and 11.2.1]{Loday}.

Let $W$, $V$ be two $\mathscr{P}_{\infty}$-algebras. Then an \emph{$\infty$-morphism} $f:W\rightarrow V$ is a
$\dg$-$\mathbb{N}$-module morphism $\mathscr{P}^{\mbox{!`} }\rightarrow\mathop{\End}_V^W$, where
$\mathop{\End}_V^W(n)=\Hom(W^{\otimes n},V)$. Moreover,
$f$ is an \emph{$\infty$-quasi-isomorphism} if $f$ sends the counit in $\mathscr{P}^{\mbox{!`} }$
to a quasi-isomorphism in $\mathop{\End}_V^W(1)$.

\subsection{Homotopy operadic algebras}

Let $(W,d_W)$ and $(V,d_V)$ be chain complexes that are $\dg$-$k$-modules. Assume $i$ and $p$ are chain maps and $h$ is chain homotopy as in the diagram\\
\centerline{\xymatrix{
*+[r]{(V,d_V)} \ar@(dl,ul)[]^{h} \ar@<+.75ex>[rr]^-p
&& (W,d_W) \ar@<+.75ex>[ll]^-{i} }.}
$W$ is a \emph{homotopy retract} of $V$ if $\Id_V-{i\circ p}=d_V\circ h+h\circ d_V $ and $i$ is a quasi-isomorphism.
Moreover, $W$ is a \emph{deformation retract} of $V$ if
we also have $ \Id_W=p\circ i$.

\begin{theorem}\cite[Theorem~10.3.1]{Loday}\label{HTT}
Let $ \mathscr{P}$ be a Koszul operad
and $(W,d_W)$ a homotopy retract of $(V,d_V)$. Any $\mathscr{P}_{\infty}$-algebra
structure on $V$ can be transferred to a $\mathscr{P}_{\infty}$-algebra structure on $W$ such that $i$ extends to an
${\infty}$-quasi-isomorphism.
\end{theorem}

This theorem, known as the Homotopy Transfer Theorem, is a generalization of \cite[Theorem~1]{Kade} and will be used in Sections
 \ref{sect:MinimalHirschBrown} and \ref{secMult}
to construct minimal Hirsch-Brown models
and minimal models discussed by Carlsson. In these constructions, we also use the following property of the bar construction:

\begin{theorem}\cite[Proposition~11.2.3]{Loday}\label{barquasi}
Let $\varphi:\mathscr{C}\rightarrow \mathscr{P}$ be an operadic twisting morphism and $A$, $A'$ $\dg$-$\mathscr{P}$-algebras.
If $f: A\rightarrow A'$ is a quasi-isomorphism, then
$f$ induces a quasi-isomorphism between the $\dg$-$\mathscr{C}$-coalgebras $\barc_{\varphi}A$ and $\barc_{\varphi}A'$.
\end{theorem}

The bar and cobar constructions form adjoint functor pair.

\begin{proposition}\cite[Corollary~11.3.5]{Loday}\label{barcobar}
Let $\mathscr{P}$ be a Koszul operad with canonical twisting morphism
$\kappa:\mathscr{P}^{\mbox{!`} }\rightarrow \mathscr{P}$. For every $\dg$-$\mathscr{P}$-algebra $A$, the counit of the adjunction
$$\epsilon_{\kappa}:\Omega_{\kappa}\barc_{\kappa}A\rightarrow A$$ is a quasi-isomorphism of $\dg$-$\mathscr{P}$-algebras.
Dually, for every conilpotent $\dg$-$\mathscr{P}^{\mbox{!`} }$-coalgebra $C$, the unit of the adjunction
$$\nu_{\kappa}:C\rightarrow \barc_{\kappa}\Omega_{\kappa}C$$
is a quasi-isomorphism of $\dg$-$\mathscr{P}^{\mbox{!`} }$-coalgebras.
\end{proposition}

The relation between $\infty$-quasi-isomorphisms and quasi-isomorphisms is given by the following:
\begin{theorem}\cite[Theorem~11.4.9]{Loday} \label{zigzag}
Let $\mathscr{P}$ be a Koszul operad and $A$, $A'$ $\dg$-$\mathscr{P}_{\infty}$-algebras.
 There exists an $\infty$-quasi-isomorphism of $\dg$ $\mathscr{P}_{\infty}$-algebras $A\rightarrow A'$
  if and only if there exists a zigzag of quasi-isomorphisms of $\dg$-$\mathscr{P}_{\infty}$-algebras
 $A\leftarrow\bullet\rightarrow\bullet \leftarrow \bullet \rightarrow\bullet \ldots\rightarrow A'$.
\end{theorem}
\noindent Such a zigzag of quasi-isomorphism will be written $A\leftrightsquigarrow A'$.
\subsection{The Poincar\'{e}-Birkhoff-Witt basis }\label{PBWbasis} We already defined Koszul duality for an operad by using bar construction. The  non-derived Koszul duality was introduced by Priddy \cite{Priddy} for algebras and generalized by Hoffbeck \cite{Hoffbeck} for operads by considering the existence of a certain basis, called Poincar\'{e}-Birkhoff-Witt ($\PBW$) basis. Hoffbeck's criterion asserts that an operad is Koszul if it admits a PBW basis.

In order to define a $\PBW$ basis for an operad $\mathscr{P}$, we consider the path sequence of the tree monomials of $\mathscr{P}$ as in \cite[Definition~3.4.1.2]{Dotsenko}. Then we order the path sequences corresponding to the trees by the graded path lexicographic order  \cite[Definition~3.4.1.7]{Dotsenko}. Briefly, given two path sequences $p$ and $ p'$ of the same arity, we have $p\prec p'$ if and only if either
\begin{itemize}
\item[(i)] the longer sequence is bigger,
\item[ ] or
\item[(ii)] if they have the same length, then we compare the first (leftmost) letters where they differ.
\end{itemize}
For instance, suppose that we have tree monomials \hspace{5pt}
$ \begin{tikzpicture}[scale=0.8,baseline=1mm]
\draw [thick](0,0.57) -- (0,0.8);
\draw [thick](0,0.02) -- (0,-0.3);
\draw  node at (0,0.3) [draw,circle, scale=0.7]{a};
\end{tikzpicture} ,
 \begin{tikzpicture}[scale=0.8,baseline=-1mm]
\draw [thick](-0.24,0.2) -- (-0.5,0.5);
\draw [thick](0.24,0.2) -- (0.5,0.5);
\draw [thick](0,-0.3) -- (0,-0.7);
\draw  node at (0,0) [draw,circle, scale=0.7]{b};
\end{tikzpicture},
\begin{tikzpicture}[scale=0.8,baseline=-1mm]
\draw [thick](-0.2,0.18) -- (-0.5,0.5);
\draw [thick](0.2,0.18) -- (0.5,0.5);
\draw [thick](0,-0.26) -- (0,-0.7);
\draw  node at (0,0) [draw,circle, scale=0.7]{c};
\end{tikzpicture}
$
equipped with a monomial order $a\prec b \prec c$. Let us consider the following tree monomials:
$$
 \begin{tikzpicture}[scale=0.8,baseline=-1mm]
\draw [thick](-0.24,0.2) -- (-0.5,0.5);
\draw [thick](0.24,0.2) -- (0.7,0.7);
\draw [thick](0,-0.3) -- (0,-0.7);
\draw  node at (0,0) [draw,circle, scale=0.7]{b};
\draw  node at (-0.53,0.76) [draw,circle, scale=0.7]{ a };
\draw [thick](-0.53,1.02) -- (-0.53,1.3);
\end{tikzpicture},\hspace{7pt}
\begin{tikzpicture}[scale=0.8,baseline=-1mm]
\draw [thick](-0.24,0.2) -- (-0.5,0.5);
\draw [thick](0.24,0.2) -- (0.5,0.5);
\draw [thick](0,-0.3) -- (0,-0.7);
\draw  node at (0,0) [draw,circle, scale=0.7]{b};
\draw  node at (-0.53,0.76) [draw,circle, scale=0.7]{ a };
\draw [thick](-0.53,1.02) -- (-0.53,1.3);
\draw  node at (0.53,0.76) [draw,circle, scale=0.7]{ a };
\draw [thick](0.53,1.02) -- (0.53,1.3);
\end{tikzpicture}\mbox{ and }
\begin{tikzpicture}[scale=0.8,baseline=-1mm]
\draw [thick](-0.24,0.2) -- (-0.5,0.5);
\draw [thick](0.24,0.2) -- (0.7,0.7);
\draw [thick](0,-0.3) -- (0,-0.7);
\draw  node at (0,0) [draw,circle, scale=0.7]{b};
\draw  node at (-0.53,0.76) [draw,circle, scale=0.7]{ c };
\draw [thick](-0.76,0.86) -- (-1.15,1.3);
\draw [thick](-0.3,0.86) -- (0.15,1.3);
\end{tikzpicture},\hspace{7pt}
\begin{tikzpicture}[scale=0.8,baseline=-1mm]
\draw [thick](-0.22,0.17) -- (-0.47,0.47);
\draw [thick](0.22,0.17) -- (0.7,0.7);
\draw [thick](0,-0.28) -- (0,-0.7);
\draw  node at (0,0) [draw,circle, scale=0.7]{c};
\draw  node at (-0.53,0.76) [draw,circle, scale=0.7]{ b };
\draw [thick](-0.8,0.88) -- (-1.15,1.3);
\draw [thick](-0.27,0.88) -- (0.15,1.3);
\end{tikzpicture}.$$
The path sequences correspond to the tree monomials have the order $(ba,b)\prec (ba,ba)$ and $(bc,bc,b)\prec (cb,cb,c)$. In other words, $(b;a,1)\prec (b;a,a)$ and $(b;c,c,1)\prec (c;b,b,1)$.
For more details, see \cite[Chapter~2.3, 3.4]{Dotsenko}.

For the next notion, we refer the reader to \cite{Hoffbeck}. For every tree $\tau$, let $\mathcal{B}^{\mathscr{T}(M)}_{\tau}$ be a monomial basis of ${\tau}(M)$ such that each element is a tensor product of elements in $\mathcal{B}^M$. A $\PBW$ basis of a non-symmetric quadratic operad $\mathscr{P}$ is a set $\mathcal{B}^{\mathscr{P}}\subset \mathscr{T}(M)$ of representatives of a base of the module $\mathscr{P}$, containing $1$ and $\mathcal{B}^M$ and for all tree $\tau$ a subset
$\mathcal{B}^{\mathscr{P}}_{\tau}\subset \mathcal{B}^{\mathscr{T}(M)}_{\tau}$ satisfying the following conditions:
\begin{itemize}
	\item  For $\alpha\in \mathcal{B}^{\mathscr{P}}_{\sigma}$ and $\beta\in \mathcal{B}^{\mathscr{P}}_{\tau}$, either the partial composition product $\alpha \circ_i \beta$ is in $\mathcal{B}^{\mathscr{P}}_{\sigma \circ_i \tau}$ or the elements of the basis $\gamma\in \mathcal{B}^{\mathscr{P}}$ which appear in the unique decomposition  $\alpha \circ_i \beta=\Sigma_{\gamma}\ c_{\gamma}\gamma$, satisfy $\gamma\succ \alpha \circ_i \beta$ in $\mathscr{T}(M)$.
	\item Suppose that $\alpha\vert_{\tau_e}$ denotes the restriction of a treewise tensor $\alpha$ to a subtree $\tau_e$ generated by an edge $e$; in other words $\alpha\vert_{\tau_e}$ is the smallest piece of tree that includes the edge $e$ and so it has $2$ vertices.  A treewise tensor $\alpha$ is in $\mathcal{B}^{\mathscr{P}}_{\tau}$ if and only if for every internal edge $e$ of $\tau$, the restricted treewise tensor $\alpha\vert_{\tau_e}$ lie in $\mathcal{B}^{\mathscr{P}}_{\tau_e}$
\end{itemize}
Moreover, by the second condition, a treewise tensor is in the basis if and only if every
subtensor generated by an edge is in the basis. Hence, it is enough to set
the quadratic part of the basis to determine the basis completely. Then we have the following result:
\begin{theorem} \label{HofbeckThm66} \cite[Theorem~6.6]{Hoffbeck} A non-symmetric operad endowed with a non-symmetric $PBW$ basis is Koszul.
\end{theorem}
We use this fact in the proof of Koszulness of the operad in Theorem \ref{mainthm}.
\section{Minimal Models}
\label{sect:MinimalHirschBrown}

In this section, $r$ denotes a positive integer. Here we discuss Hirsch-Brown Models in view of the Homotopy Transfer Theorem.
\subsection{Unary quadratic (co)operads}
Let $(M,R)$ be the quadratic data pair
$$M=(\,0\, ,\, k\textup{v}_1\oplus k\textup{v}_2\oplus \dots \oplus
k\textup{v}_r \, ,\, 0\, ,\, 0\, ,\, \dots\, )$$
and
$$R=
\{\,\,\textup{v}_i\otimes \textup{v}_i\,\,|\,\,1\leq i \leq r\,\, \}
\cup
\{\,\,\textup{v}_i\otimes \textup{v}_j+\textup{v}_j\otimes \textup{v}_i\,\,|\,\,1\leq i<j\leq r\,\,\}.
$$
We define a quadratic operad $\mathscr{W}$ and a quadratic cooperad $\mathscr{S}$ as follows:
$$\mathscr{W}:=\mathscr{P}(M,R)\text{\ \ \ \ and \ \ \ \ }\mathscr{S}:=\mathscr{C}(sM,s^2R).$$
Then considering the identifications
\begin{align*}
t_i=
\hspace{7pt}
\begin{tikzpicture}[scale=0.8,baseline=-1mm]
\draw [thick](0,0.66) -- (0,1);
\draw [thick](0,-0.16) -- (0,-0.5);
\draw  node at (0,0.25) [draw,circle, scale=1]{\small{v}$_i$};
\end{tikzpicture}
\hspace{7pt} \hspace{7pt}\text{ and \ \ \ }
x_i^*=
\hspace{7pt}
\begin{tikzpicture}[scale=0.8,baseline=-1mm]
\draw [thick](0,0.72) -- (0,1.1);
\draw [thick](0,-0.22) -- (0,-0.6);
\draw  node at (0,0.25) [draw,circle, scale=1]{\small{sv}$_i$};
\end{tikzpicture}
\end{align*}
for $i$ in $\{1,2,\dots r\}$, the operad $\mathscr{W}$ is isomorphic to
$(0,\Lambda,0,0, \dots )$  where $\Lambda $ is the exterior algebra $\Lambda(t_1,t_2,\ldots,t_r)$
and the cooperad $\mathscr{S}$ is isomorphic to
$(0,S^*,0,0, \dots )$ where $S$ is the polynomial algebra $k[x_1,\ldots,x_r]$.
\subsection{Minimal Hirsch-Brown models}\label{minmodelsection}
 First note that we consider cochain complexes as chain complexes after multiplying the grading by $-1$. In other words, there exists a categorical isomorphism between the category of chain complexes and the category of cochain complexes by identifying a chain complex $(C,\partial )$ and a cochain complex $(D,\zeta )$ if $C_{-i}=D^i$ and $\partial_{-i}=\zeta ^{i} $. From now on we only work with chain complexes.  The cohomology of a simplicial set $X$ will be denoted by
  $ H^{\bullet}(X;k)$ which corresponds to $ \bigoplus_{m=0}^{\infty} H^{m}(X;k)$ under the isomorphism mentioned above. Hence, $ H^{\bullet}(X;k)$ is trivial in all positive degrees.

 Let $G$ be an elementary abelian $2$-group of rank $r$ with generator set $\{g_1,\ldots,g_r\}$. Then we can identify the group algebra $kG$ with the exterior algebra $\Lambda$ by identifying $1+g_i$ with $t_i$.
 Moreover, the group cohomology $H^{\bullet}(BG;k)$ is isomorphic to the polynomial algebra $S$ as graded $k$-algebras.  Assume that $G$ acts freely on a simplicial set $X$. The goal of this section is to use the techniques discussed in the previous section to give a construction of the minimal Hirsch-Brown model of $X$ which is equivalent to the one constructed in \cite{AlldayPuppe}. In other words, we define a  differential graded $S$-module denoted by $H^{\bullet}(BG;k)\widetilde{\otimes}H^{\bullet}(X;k)$ so that $H^{\bullet}(BG;k)\widetilde{\otimes}H^{\bullet}(X;k)$  is isomorphic to  $H^{\bullet}(BG;k){\otimes}H^{\bullet}(X;k)$ as a left $S$-module, and there exists a zig zag of quasi-isomorphisms  between $H^{\bullet}(BG;k)\widetilde{\otimes}H^{\bullet}(X;k)$ and a differential graded $S$-module which is chain homotopy equivalent to the cochain complex of the Borel construction $EG\times_G X$.

The chain complex $C=C(X;k)$  is a dg-$\mathscr{W}$-algebra by the morphism from $kG\otimes C$ to $C$ which sends
 $g\otimes \sigma $ to $g\sigma $ for any $g\in G$, $\sigma \in C$. Moreover, we have $\mathscr{S}=\mathscr{W}^{\mbox{!`} }$ and $\mathscr{W}_{\infty}=\Omega\mathscr{S}$. Let $j$ denote the inclusion of $\dg$-$\mathscr{W}$-algebras into $\dg$-$\Omega\mathscr{S}$-algebras.
 Since $H(C)$ is a deformation retract of $j(C)$ as $\dg$-$k$-modules,
by Theorem \ref{HTT} there exists a $\Omega\mathscr{S}$-algebra structure on $H(C)$ such that $H(C)$ and $j(C)$ are $\infty $-quasi-isomorphic as
$\Omega\mathscr{S}$-algebras. We know that $\mathscr{W}$ is also a Koszul operad. Then by Theorem \ref{zigzag}, there exists a
 zigzag of quasi-isomorphisms of $\dg$-$\Omega\mathscr{S}$-algebras $H(C)\leftrightsquigarrow j(C)$.

  Note that $\mathscr{S}$ is a connected cooperad, so it is conilpotent.
  Let ${\iota:\mathscr{S}\rightarrow \Omega\mathscr{S}}$ be the universal twisting morphism.
  By Theorem \ref{barquasi}, there is a zigzag of quasi-isomorphisms
  $\barc_{\iota} H(C)\leftrightsquigarrow (\barc_{\iota}\circ j)(C)$ as $\mathscr{S}$-coalgebras.
  As graded $\mathbb{N}$-modules, we have the following isomorphism:
  $$\barc_{\iota} H(C)=\left(\mathscr{S}\circ_{\iota}\Omega \mathscr{S}\right)\circ_{\Omega \mathscr{S}}H(C)\cong S^*{\otimes} H(C).$$
This isomorphism induces a differential on $ S^*{\otimes} H(C)$. We denote the new differential graded $\mathbb{N}$-module by $S^*\widetilde{\otimes} H(C)$.

We consider the ${\mathscr{S}^*}$-algebra $(\barc_{\iota} H(C))^*$ as a version of the minimal Hirsch-Brown model because
\begin{align*}
(\barc_{\iota} H(C))^*&\cong(S^*\widetilde{\otimes}H(C))^*\\
 &\cong S\widetilde{\otimes}(H(C))^*\\
 &\cong H^{\bullet}(BG;k)\widetilde{\otimes}H^{\bullet}(X;k).
\end{align*}
%where $\widetilde{\otimes}$ is the usual tensor product of $k$-modules with a possibly twisted differential, see \cite{AlldayPuppe}.

Let $\kappa: \mathscr{S}\rightarrow \mathscr{W}$ be the canonical twisting morphism.
Note that $C(EG;k)$ is $kG$-chain homotopy equivalent to
$S^*\widetilde{\otimes}kG:=\mathscr{S}\circ_{\kappa}\mathscr{W}$,
where both are considered as differential graded right $kG$-modules.
 Also we have $(\barc_{\iota}\circ j)= \barc_{\kappa}$. Hence
\begin{align*}
((\barc_{\iota}\circ j)(C))^*&\cong (S^*\widetilde{\otimes}kG \otimes_{kG}C)^*\\
&\simeq (C(EG;k)\otimes_{kG}C)^*\\
&\simeq C(EG\times_G X;k)^* \\
&= C^{\bullet}(EG\times_G X;k)
\end{align*}
where the last equality is due to our conventions about cochain complexes and the second homotopy equivalence follows from the homotopy equivalence $C(EG\times_G X;k) \simeq C(EG;k)\otimes_{kG} C$ proved in \cite[proof of Theorem~1.2.8]{AlldayPuppe} and \cite[\RN{6}.12]{Dold}.

We can also consider the chain complex $C=C(X;k)$  as a dg-$\mathscr{W}^*$-coalgebra by the morphism from $C$ to $kG\otimes C$ which send $\sigma $ to $\sum _{g\in G}g^*\otimes g^{-1}\sigma $. Hence $C^*$  is a dg-$\mathscr{W}$-algebra as discussed in Section \ref{coalgebras}.
Hence the $\mathscr{S}$-coalgebra $\barc_{\iota} H(C^*)$ is another version of the minimal Hirsch-Brown model. In fact this second version is what we use in Section \ref{secMult}.

\subsection{The Minimal model of Carlsson}\label{minCarls}
Let $N$ be a differential graded $S$-module, so it is a $\dg$-${\mathscr{S}}^*$-algebra. We view $N$ as a $\dg$-$\mathscr{S}$-coalgebra. The goal of this section is to construct Carlsson's minimal model \cite{Carlsson2} for $N$. We construct a $\dg$-$\mathscr{S}$-coalgebra
that is quasi-isomorphic to $N$ and has zero differential when tensored with $k$ over $\mathscr{S}$.

We have $F=F_2(N)$ in the filtration from Section \ref{coalgebras}, so the coalgebra $N$ is conilpotent.
 As a $\dg$-$k$-module $H(N)$ is a deformation retract of $N$.
 We obtain the following deformation retract of $\dg$-$k$-modules by applying
 the functor $\Omega_{\kappa}$, where $\kappa:\mathscr{S}\rightarrow \mathscr{S}^{\mbox{!`} } $
  is the canonical twisting morphism

 \centerline{\xymatrix{
*+[r]{(\Omega_{\kappa}N)} \ar@(dl,ul)[]^{\Omega_{\kappa}(h)} \ar@<+.75ex>[rr]^-{\Omega_{\kappa}(p)}
&& (\Omega_{\kappa}\Homology(N)) \ar@<+.75ex>[ll]^-{\Omega_{\kappa}(i)} }.}

By Theorem \ref{HTT},
$\Omega_{\kappa}N \xleftarrow{\Omega_{\kappa}(i)} \Omega_{\kappa}\Homology(N)$
extends to an $\infty$-quasi-isomorphism of $\dg$-$\Omega\mathscr{S}$-algebras.
Furthermore, we have another deformation retract

 \centerline{\xymatrix{
*+[r]{(\Omega_{\kappa}\Homology(N))} \ar@(dl,ul)[]^{(h')} \ar@<+.75ex>[rr]^-{p'}
&& (H(\Omega_{\kappa}\Homology(N))) \ar@<+.75ex>[ll]^-{i'} }}
\noindent and so
$\Omega_{\kappa}\Homology(N)\xleftarrow{i'} \Homology(\Omega_{\kappa}\Homology(N))$ extends to an
$\infty$-quasi-isomorphism of $\dg$-$\Omega\mathscr{S}$-algebras by Theorem \ref{HTT}.
Combining these two $\infty$-quasi-isomorphisms, we have an
 $\infty$-quasi-isomorphism
of $\dg$-$\Omega\mathscr{S}$-algebras $\Omega_{\kappa}N\leftarrow \Homology(\Omega_{\kappa}\Homology(N))$.
Thus by Theorem \ref{zigzag}, there is a zigzag of quasi-isomorphisms as
$\dg$-$\Omega\mathscr{S}$-algebras
$$\Omega_{\kappa}N \leftrightsquigarrow  \Homology(\Omega_{\kappa}H(N)).$$
Then by Theorem \ref{barquasi}, we have a zigzag of quasi-isomorphisms of $\dg$-$\mathscr{S}$-coalgebras
$$\barc_{\kappa}\Omega_{\kappa}N \leftrightsquigarrow \barc_{\kappa}\Homology(\Omega_{\kappa}\Homology(N)).$$
There is a quasi-isomorphism of $\dg$-$\mathscr{S}$-coalgebras $N\rightarrow \barc_{\kappa}\Omega_{\kappa}N$
 by Proposition \ref{barcobar}. Therefore,
we obtain a zigzag of quasi-isomorphisms of $\dg$-$\mathscr{S}$-coalgebras
 $$N \leftrightsquigarrow \barc_{\kappa}\Homology(\Omega_{\kappa}\Homology(N)).$$
Note that $k\otimes_{S}\barc_{\kappa}\Homology(\Omega_{\kappa}\Homology(N))$ has zero differential. Hence we call the $\dg$-$\mathscr{S}$-coalgebra
$\barc_{\kappa}\Homology(\Omega_{\kappa}\Homology(N))$ the Carlsson minimal model of $N$.

\subsection{A special case of Carlsson's conjecture}
The following is equivalent to Theorem \ref{thm1}:
\begin{theorem}
 Let $k$ be an algebraically closed field of characteristic $2$ and $S$ the polynomial algebra in $r$ variables of degree $-1$ with coefficients in $k$. Assume $(M,\partial)$ is a free $\dg$-$S$-module and $0<\dim_k	\Homology(M)< \infty $.
Further assume that $\chi(\Homology(M)):=\sum\limits_{i\geq 0}(-1)^i \dim_k{\Homology_i(M)}$ is non-zero. Then $2^r\leq \rank_S M$.
\end{theorem}
\begin{proof}
We can consider $M$ as a $\dg$-$\mathscr{S}$-coalgebra.
As in Section \ref{minCarls},
 we have a zigzag of quasi-isomorphism of $\dg$-$\mathscr{S}$-coalgebras
$$
M\leftrightsquigarrow \barc_{\kappa}\Homology(\Omega_{\kappa}\Homology(M)),
$$
where each middle term in this zigzag is free.

If $f: K\rightarrow L$ is a quasi-isomorphism of bounded-below complexes of free modules,
then the mapping cone of $f$ is a bounded-below acyclic complex of free modules.
  Therefore, the mapping cone is contractible and $f$ is split, so $f$ is a homotopy equivalence \cite[Proposition~0.3,
Proposition~0.7]{Brown}. This implies the following zigzag of quasi-isomorphism:
$$
k\otimes_{S}M\leftrightsquigarrow k
\otimes_{S}\barc_{\kappa}\Homology(\Omega_{\kappa}\Homology(M))\cong \Homology(\Omega_{\kappa}H(M)).
$$
Also notice that
$$\chi(\Homology(\Omega_{\kappa}\Homology(M)))=\chi(\Omega_{\kappa}\Homology(M))=2^r\chi(\Homology(M))\neq 0,$$
Thus,
$$2^r\leq  \dim_k(\Homology(\Omega_{\kappa}\Homology(M)))= \dim_k(\Homology(k\otimes_S M))\leq \dim_k(k\otimes_S M)=\rank_S(M).$$
\end{proof}
\section{Multiplicative Structures on Minimal Hirsch-Brown Models}\label{secMult}  %%Proof of Theorem \ref{mainthm}
In this section, we will prove that the operad $\tilde{\mathscr{W}}$ defined in Section \ref{ssBV} is an operad that satisfies the properties listed in Theorem \ref{mainthm}.
\subsection{The operad $\tilde{\mathscr{W}}$ in the case $r=1$}\label{ssBVr1}
Consider the associative operad that is generated by a binary operation $\mu_0$, that satisfies the associativity relation $(\mu_0;\mu_0,1) = (\mu_0;1, \mu_0).$
In terms of trees, we have
\[
\begin{tikzpicture}[scale=0.9,baseline=-0.5mm]
	\draw[fill] (0,0) node[ left] {$\mu_0$} circle [radius=0.07];
\end{tikzpicture}\hspace{5pt}=
\begin{tikzpicture}[scale=0.9,baseline=-1mm]
	\draw [thick](-0.2,0.2) -- (-0.5,0.5);
	\draw [thick](0.2,0.2) -- (0.5,0.5);
	\draw [thick](0,-0.3) -- (0,-0.7);
	\draw  node at (0,0) [draw,circle, scale=0.7]{${\mu_0}$};
\end{tikzpicture} \text{ with the relation }
\begin{tikzpicture}[scale=0.9,baseline=-0.5mm]
	\draw [thick](0,-0.4) node[below left] {$\mu_0$} -- (-0.6,0.3) node[ left] {$\mu_0$};
	\draw[fill] (-0.6,0.3) circle [radius=0.07];
	\draw[fill] (0,-0.4) circle [radius=0.07];
\end{tikzpicture} \quad =
\begin{tikzpicture}[scale=0.9,baseline=-0.5mm]
	\draw [thick](0,-0.4) node[below left] {$\mu_0$} -- (0.6,0.3) node[ left] {$\mu_0$};
	\draw[fill] (0.6,0.3) circle [radius=0.07];
	\draw[fill] (0,-0.4) circle [radius=0.07];
\end{tikzpicture}.
\]
Similarly, for an exterior algebra of a single variable, we have an unary operation $t$;
$$
\begin{tikzpicture}[scale=0.9,baseline=-0.5mm]
	\draw[fill] (0,0)  node[ left] {t}circle [radius=0.07];
\end{tikzpicture}\hspace{5pt}= \,
\begin{tikzpicture}[scale=0.9,baseline=-0.5mm]
	\draw [thick](0,0.57) -- (0,0.8);
	\draw [thick](0,0.03) -- (0,-0.3);
	\draw  node at (0,0.3) [draw,circle, scale=0.7]{\Large{t}};
\end{tikzpicture} \text{ with the relation }
\begin{tikzpicture}[scale=0.9,baseline=-0.5mm]
	\draw [thick](0,0) -- (0,0.6) node[ left] {t};
	\draw [thick](0,0) -- (0,-0.3) node[ left] {t};
	\draw[fill] (0,0.6) circle [radius=0.07];
	\draw[fill] (0,-0.3) circle [radius=0.07];
\end{tikzpicture}\, = \,0.
$$
In the case $r=1$, we define a quadratic operad $\tilde{\mathscr{W}}$ by setting generating operations as
$$\begin{tikzpicture}[scale=0.9,baseline=-0.5mm]
	\draw[fill] (0,0)  node[ left] {t}circle [radius=0.07];
\end{tikzpicture}=
\begin{tikzpicture}[scale=0.9,baseline=-0.5mm]
	\draw [thick](0,0.27) -- (0,0.55);
	\draw [thick](0,-0.27) -- (0,-0.65);
	\draw  node at (0,0) [draw,circle, scale=0.7]{\Large{t}};
\end{tikzpicture}\hspace{5pt}, \hspace{2pt}
\begin{tikzpicture}[scale=0.9,baseline=-0.5mm]
	\draw[fill] (0,0) node[ left] {$\mu_0$} circle [radius=0.07];
\end{tikzpicture}=\begin{tikzpicture}[scale=0.9,baseline=-1mm]
	\draw [thick](-0.2,0.2) -- (-0.5,0.5);
	\draw [thick](0.2,0.2) -- (0.5,0.5);
	\draw [thick](0,-0.3) -- (0,-0.7);
	\draw  node at (0,0) [draw,circle, scale=0.7]{${\mu_0}$};
\end{tikzpicture} \hspace{5pt}\text{ and } \hspace{2pt}
\begin{tikzpicture}[scale=0.9,baseline=-0.5mm]
	\draw[fill] (0,0) node[ left] {${\mu_1}$} circle [radius=0.07];
\end{tikzpicture}\hspace{2pt}=
\begin{tikzpicture}[scale=0.9,baseline=-1mm]
	\draw [thick](-0.2,0.2) -- (-0.5,0.5);
	\draw [thick](0.2,0.2) -- (0.5,0.5);
	\draw [thick](0,-0.3) -- (0,-0.7);
	\draw  node at (0,0) [draw,circle, scale=0.7]{${\mu_1}$};
\end{tikzpicture}
$$

\noindent The relations of $\tilde{\mathscr{W}}$ are as follows :
\\
\\
\noindent $R_1$: \hspace{10pt}
\begin{tikzpicture}[scale=0.9,baseline=-0.5mm]
	\draw [thick](0,0) node[below left] {t} -- (0,0.3);
	\draw [thick](0,0) node[above left] {t}-- (0,-0.4);
	\draw[fill] (0,0.3) circle [radius=0.07];
	\draw[fill] (0,-0.4) circle [radius=0.07];
\end{tikzpicture} $\quad = 0$, \hspace{30pt}
$R_2 : $ \hspace{10pt}
\begin{tikzpicture}[scale=0.9,baseline=-0.5mm]
	\draw [thick](0,-0.4) node[below left] {${\mu_1}$} -- (-0.6,0.3) node[ left] {t};
	\draw[fill] (-0.6,0.3) circle [radius=0.07];
	\draw[fill] (0,-0.4) circle [radius=0.07];
\end{tikzpicture} $\quad = 0$\hspace{10pt}, %\text{ ($t$ satisfies the exterior product rule), }
\\
\\
\\
$R_3 : $ \hspace{10pt}
\begin{tikzpicture}[scale=0.9,baseline=-0.5mm]
	\draw [thick](0,-0.4) node[below left] {$\mu_0$} -- (-0.6,0.3) node[ left] {$\mu_0$};
	\draw[fill] (-0.6,0.3) circle [radius=0.07];
	\draw[fill] (0,-0.4) circle [radius=0.07];
\end{tikzpicture} $\quad =$
\begin{tikzpicture}[scale=0.9,baseline=-0.5mm]
	\draw [thick](0,-0.4) node[below left] {$\mu_0$} -- (0.6,0.3) node[ left] {$\mu_0$};
	\draw[fill] (0.6,0.3) circle [radius=0.07];
	\draw[fill] (0,-0.4) circle [radius=0.07];
\end{tikzpicture}, \hspace{30pt}
$R_4 : $ \hspace{10pt}
\begin{tikzpicture}[scale=0.9,baseline=-0.5mm]
	\draw [thick](0,-0.4) node[below left] {${\mu_1}$} -- (-0.6,0.3) node[ left] {${\mu_1}$};
	\draw[fill] (-0.6,0.3) circle [radius=0.07];
	\draw[fill] (0,-0.4) circle [radius=0.07];
\end{tikzpicture} $\quad =$
\begin{tikzpicture}[scale=0.9,baseline=-0.5mm]
	\draw [thick](0,-0.4) node[below left] {${\mu_1}$} -- (0.6,0.3) node[ left] {${\mu_1}$};
	\draw[fill] (0.6,0.3) circle [radius=0.07];
	\draw[fill] (0,-0.4) circle [radius=0.07];
\end{tikzpicture} \hspace{10pt}, %\text{ ($\mu_0$ and ${\mu_1}$ are associative),}
\\
\\
\\
$R_5$: \hspace{10pt}
\begin{tikzpicture}[scale=0.9,baseline=-0.5mm]
	\draw [thick](0,0) node[below left] {t} -- (0,0.3);
	\draw [thick](0,0) node[above left] {$\mu_0$}-- (0,-0.4);
	\draw[fill] (0,0.3) circle [radius=0.07];
	\draw[fill] (0,-0.4) circle [radius=0.07];
\end{tikzpicture} \quad $=$
\begin{tikzpicture}[scale=0.9,baseline=-0.5mm]
	\draw [thick](0,-0.4) node[below left] {$\mu_0$} -- (-0.6,0.3) node[ left] {t};
	\draw[fill] (-0.6,0.3) circle [radius=0.07];
	\draw[fill] (0,-0.4) circle [radius=0.07];
\end{tikzpicture} \quad $ +$
\begin{tikzpicture}[scale=0.9,baseline=-0.5mm]
	\draw [thick](0,-0.4) node[below left] {$\mu_0$} -- (0.6,0.3) node[ left] {t};
	\draw[fill] (0.6,0.3) circle [radius=0.07];
	\draw[fill] (0,-0.4) circle [radius=0.07];
\end{tikzpicture}\quad  $ +$
\begin{tikzpicture}[scale=0.9,baseline=-0.5mm]
	\draw [thick](0,-0.4) node[below left] {${\mu_1}$} -- (0.6,0.3) node[ left] {$t$};
	\draw[fill] (0.6,0.3) circle [radius=0.07];
	\draw[fill] (0,-0.4) circle [radius=0.07];
\end{tikzpicture} \hspace{10pt},
\\
\\
\\
$R_6 : $ \hspace{10pt}
\begin{tikzpicture}[scale=0.9,baseline=-0.5mm]
	\draw [thick](0,-0.4) node[below left] {${\mu_0}$} -- (-0.6,0.3) node[ left] {${\mu_1}$};
	\draw[fill] (-0.6,0.3) circle [radius=0.07];
	\draw[fill] (0,-0.4) circle [radius=0.07];
\end{tikzpicture} $\quad =$
\begin{tikzpicture}[scale=0.9,baseline=-0.5mm]
	\draw [thick](0,-0.4) node[below left] {${\mu_1}$} -- (0.6,0.3) node[ left] {${\mu_0}$};
	\draw[fill] (0.6,0.3) circle [radius=0.07];
	\draw[fill] (0,-0.4) circle [radius=0.07];
\end{tikzpicture}, \hspace{30pt}
$R_7$: \hspace{10pt}
\begin{tikzpicture}[scale=0.9,baseline=-0.5mm]
	\draw [thick](0,0) node[below left] {t} -- (0,0.3);
	\draw [thick](0,0) node[above left] {${\mu_1}$}-- (0,-0.4);
	\draw[fill] (0,0.3) circle [radius=0.07];
	\draw[fill] (0,-0.4) circle [radius=0.07];
\end{tikzpicture} \quad $=$
\begin{tikzpicture}[scale=0.9,baseline=-0.5mm]
	\draw [thick](0,-0.4) node[below left] {${\mu_1}$} -- (0.6,0.3) node[ left] {t};
	\draw[fill] (0.6,0.3) circle [radius=0.07];
	\draw[fill] (0,-0.4) circle [radius=0.07];
\end{tikzpicture} \hspace{10pt},
\\
\\
\\
$R_8 : $ \hspace{10pt}
\begin{tikzpicture}[scale=0.9,baseline=-0.5mm]
	\draw [thick](0,-0.4) node[below left] {${\mu_1}$} -- (-0.6,0.3) node[ left] {$\mu_0$};
	\draw[fill] (-0.6,0.3) circle [radius=0.07];
	\draw[fill] (0,-0.4) circle [radius=0.07];
\end{tikzpicture} $\quad =$
\begin{tikzpicture}[scale=0.9,baseline=-0.5mm]
	\draw [thick](0,-0.4) node[below left] {${\mu_0}$} -- (-0.6,0.3) node[ left] {${\mu_1}$};
	\draw[fill] (-0.6,0.3) circle [radius=0.07];
	\draw[fill] (0,-0.4) circle [radius=0.07];
\end{tikzpicture} \quad $+$
\begin{tikzpicture}[scale=0.9,baseline=-0.5mm]
	\draw [thick](0,-0.4) node[below left] {$\mu_0$} -- (0.6,0.3) node[ left] {${\mu_1}$};
	\draw[fill] (0.6,0.3) circle [radius=0.07];
	\draw[fill] (0,-0.4) circle [radius=0.07];
\end{tikzpicture}
\quad $+$
\begin{tikzpicture}[scale=0.9,baseline=-0.5mm]
	\draw [thick](0,-0.4) node[below left] {${\mu_1}$} -- (0.6,0.3) node[ left] {${\mu_1}$};
	\draw[fill] (0.6,0.3) circle [radius=0.07];
	\draw[fill] (0,-0.4) circle [radius=0.07];
\end{tikzpicture}.
\\
\\
Please see Lemma \ref{Lemmarelationr1} in Section \ref{MultiplicativeStr} to understand where these relations come from.

Now consider the graded path lex order on all quadratics; firstly there is only one $1$-ary quadratic operation and it is represented by
\begin{tikzpicture}[scale=0.9,baseline=-0.5mm]
	\draw [thick](0,0) node[below left] {t} -- (0,0.3);
	\draw [thick](0,0) node[above left] {t}-- (0,-0.4);
	\draw[fill] (0,0.3) circle [radius=0.07];
	\draw[fill] (0,-0.4) circle [radius=0.07];
\end{tikzpicture}.
Secondly, we sort all $2$-ary operations:
$$
\begin{tikzpicture}[scale=0.9,baseline=-0.5mm]
	\draw [thick](0,-0.4) node[below left] {$\mu_0$} -- (0.6,0.3) node[ left] {t};
	\draw[fill] (0.6,0.3) circle [radius=0.07];
	\draw[fill] (0,-0.4) circle [radius=0.07];
\end{tikzpicture}\hspace{5pt}\prec \hspace{5pt}
\begin{tikzpicture}[scale=0.9,baseline=-0.5mm]
	\draw [thick](0,-0.4) node[below left] {${\mu_1}$} -- (0.6,0.3) node[ left] {t};
	\draw[fill] (0.6,0.3) circle [radius=0.07];
	\draw[fill] (0,-0.4) circle [radius=0.07];
\end{tikzpicture}\hspace{5pt}\prec\hspace{5pt}
\begin{tikzpicture}[scale=0.9,baseline=-0.5mm]
	\draw [thick](0,-0.4) node[below left] {t} -- (0,0.3) node[ left] {$\mu_0$};
	\draw[fill] (0,0.3) circle [radius=0.07];
	\draw[fill] (0,-0.4) circle [radius=0.07];
\end{tikzpicture}\hspace{5pt}\prec\hspace{5pt}
\begin{tikzpicture}[scale=0.9,baseline=-0.5mm]
	\draw [thick](0,-0.4) node[below left] {t} -- (0,0.3) node[ left] {${\mu_1}$};
	\draw[fill] (0,0.3) circle [radius=0.07];
	\draw[fill] (0,-0.4) circle [radius=0.07];
\end{tikzpicture}\hspace{5pt}\prec\hspace{5pt}
\begin{tikzpicture}[scale=0.9,baseline=-0.5mm]
	\draw [thick](0,-0.4) node[below left] {$\mu_0$} -- (-0.6,0.3) node[ left] {t};
	\draw[fill] (-0.6,0.3) circle [radius=0.07];
	\draw[fill] (0,-0.4) circle [radius=0.07];
\end{tikzpicture}.
$$
Correspondingly, path sequences of the planar rooted trees are
$$(\mu_0, \mu_0 t) \prec ({\mu_1}, {\mu_1} t) \prec (t \mu_0, t \mu_0)\prec (t{\mu_1}, t{\mu_1} ) \prec (\mu_0 t,\mu_0).$$
Then we sort all $3$-ary operations:
\begin{align*}
\begin{tikzpicture}[scale=0.9,baseline=-0.5mm]
	\draw [thick](0,-0.4) node[below left] {$\mu_0$} -- (0.6,0.3) node[ left] {$\mu_0$};
	\draw[fill] (0.6,0.3) circle [radius=0.07];
	\draw[fill] (0,-0.4) circle [radius=0.07];
\end{tikzpicture}\prec
\begin{tikzpicture}[scale=0.9,baseline=-0.5mm]
	\draw [thick](0,-0.4) node[below left] {$\mu_0$} -- (0.6,0.3) node[ left] {${\mu_1}$};
	\draw[fill] (0.6,0.3) circle [radius=0.07];
	\draw[fill] (0,-0.4) circle [radius=0.07];
\end{tikzpicture}\prec
\begin{tikzpicture}[scale=0.9,baseline=-0.5mm]
	\draw [thick](0,-0.4) node[below left] {${\mu_1}$} -- (0.6,0.3) node[ left] {$\mu_0$};
	\draw[fill] (0.6,0.3) circle [radius=0.07];
	\draw[fill] (0,-0.4) circle [radius=0.07];
\end{tikzpicture}\prec
\begin{tikzpicture}[scale=0.9,baseline=-0.5mm]
	\draw [thick](0,-0.4) node[below left] {${\mu_1}$} -- (0.6,0.3) node[ left] {${\mu_1}$};
	\draw[fill] (0.6,0.3) circle [radius=0.07];
	\draw[fill] (0,-0.4) circle [radius=0.07];
\end{tikzpicture}\prec
\begin{tikzpicture}[scale=0.9,baseline=-0.5mm]
	\draw [thick](0,-0.4) node[below left] {$\mu_0$} -- (-0.6,0.3) node[ left] {${\mu_0}$};
	\draw[fill] (-0.6,0.3) circle [radius=0.07];
	\draw[fill] (0,-0.4) circle [radius=0.07];
\end{tikzpicture}\prec
\begin{tikzpicture}[scale=0.9,baseline=-0.5mm]
	\draw [thick](0,-0.4) node[below left] {$\mu_0$} -- (-0.6,0.3) node[ left] {${\mu_1}$};
	\draw[fill] (-0.6,0.3) circle [radius=0.07];
	\draw[fill] (0,-0.4) circle [radius=0.07];
\end{tikzpicture}\prec
\begin{tikzpicture}[scale=0.9,baseline=-0.5mm]
	\draw [thick](0,-0.4) node[below left] {${\mu_1}$} -- (-0.6,0.3) node[ left] {${\mu_0}$};
	\draw[fill] (-0.6,0.3) circle [radius=0.07];
	\draw[fill] (0,-0.4) circle [radius=0.07];
\end{tikzpicture}\prec
\begin{tikzpicture}[scale=0.9,baseline=-0.5mm]
	\draw [thick](0,-0.4) node[below left] {${\mu_1}$} -- (-0.6,0.3) node[ left] {${\mu_1}$};
	\draw[fill] (-0.6,0.3) circle [radius=0.07];
	\draw[fill] (0,-0.4) circle [radius=0.07];
\end{tikzpicture}.
\end{align*}
Correspondingly, path sequences of the planar rooted trees are
\begin{align*}
\begin{array}{ll}
(\mu_0, \mu_0^2,\mu_0^2)\prec (\mu_0, \mu_0 {\mu_1},\mu_0 {\mu_1})\prec ( {\mu_1}, {\mu_1}\mu_0, {\mu_1}\mu_0) \prec ({\mu_1}, {{\mu_1}}^2,{{\mu_1}}^2)&\\
\prec (\mu_0^2,\mu_0^2,\mu_0)\prec (\mu_0 {\mu_1},\mu_0 {\mu_1}, \mu_0)
\prec ({\mu_1}\mu_0, {\mu_1}\mu_0, {\mu_1})\prec ({\mu_1}^2,{\mu_1}^2,{\mu_1})&.
\end{array}
\end{align*}
Hence, the quadratic part of a non-symmetric $\mathrm{PBW}$ basis is given by
$$
\begin{tikzpicture}[scale=0.9,baseline=-0.5mm]
	\draw [thick](0,-0.4) node[below left] {$\mu_0$} -- (0.6,0.3) node[ left] {t};
	\draw[fill] (0.6,0.3) circle [radius=0.07];
	\draw[fill] (0,-0.4) circle [radius=0.07];
\end{tikzpicture}\hspace{5pt},\hspace{5pt}
\begin{tikzpicture}[scale=0.9,baseline=-0.5mm]
	\draw [thick](0,-0.4) node[below left] {${\mu_1}$} -- (0.6,0.3) node[ left] {t};
	\draw[fill] (0.6,0.3) circle [radius=0.07];
	\draw[fill] (0,-0.4) circle [radius=0.07];
\end{tikzpicture}\hspace{5pt},\hspace{5pt}
\begin{tikzpicture}[scale=0.9,baseline=-0.5mm]
	\draw [thick](0,-0.4) node[below left] {t} -- (0,0.3) node[ left] {$\mu_0$};
	\draw[fill] (0,0.3) circle [radius=0.07];
	\draw[fill] (0,-0.4) circle [radius=0.07];
\end{tikzpicture}\hspace{5pt},\hspace{5pt}
\begin{tikzpicture}[scale=0.9,baseline=-0.5mm]
	\draw [thick](0,-0.4) node[below left] {$\mu_0$} -- (0.6,0.3) node[ left] {$\mu_0$};
	\draw[fill] (0.6,0.3) circle [radius=0.07];
	\draw[fill] (0,-0.4) circle [radius=0.07];
\end{tikzpicture}\hspace{5pt},\hspace{5pt}
\begin{tikzpicture}[scale=0.9,baseline=-0.5mm]
	\draw [thick](0,-0.4) node[below left] {$\mu_0$} -- (0.6,0.3) node[ left] {${\mu_1}$};
	\draw[fill] (0.6,0.3) circle [radius=0.07];
	\draw[fill] (0,-0.4) circle [radius=0.07];
\end{tikzpicture}\hspace{5pt},\hspace{5pt}
\begin{tikzpicture}[scale=0.9,baseline=-0.5mm]
	\draw [thick](0,-0.4) node[below left] {${\mu_1}$} -- (0.6,0.3) node[ left] {$\mu_0$};
	\draw[fill] (0.6,0.3) circle [radius=0.07];
	\draw[fill] (0,-0.4) circle [radius=0.07];
\end{tikzpicture}\hspace{5pt},\hspace{5pt}
\begin{tikzpicture}[scale=0.9,baseline=-0.5mm]
	\draw [thick](0,-0.4) node[below left] {${\mu_1}$} -- (0.6,0.3) node[ left] {${\mu_1}$};
	\draw[fill] (0.6,0.3) circle [radius=0.07];
	\draw[fill] (0,-0.4) circle [radius=0.07];
\end{tikzpicture}.
$$
Correspondingly, path sequences of the quadratic part of the basis is given by
\begin{align*}
\begin{array}{l}
(\mu_0, \mu_0 t)\prec ({\mu_1}, {\mu_1} t)\prec (t \mu_0, t \mu_0)\prec
(\mu_0, \mu_0^2,\mu_0^2)\prec (\mu_0, \mu_0 {\mu_1},\mu_0 {\mu_1})\\
\prec ({\mu_1}, {\mu_1}\mu_0, {\mu_1}\mu_0)\prec ( {\mu_1}, {{\mu_1}}^2,{{\mu_1}}^2).
\end{array}
\end{align*}
The other way around those trees correspond to the elements;
$$(\mu_0;1,t)\prec({\mu_1};1,t)\prec(t;\mu_0 )\prec(\mu_0;1,\mu_0)\prec(\mu_0;1,{\mu_1})\prec({\mu_1};1,\mu_0)
\prec({\mu_1};1,{\mu_1}).$$
\subsection{The operad $\tilde{\mathscr{W}}$ in general}\label{ssBV}
For a positive integer $r$, let  $(M,R)$ be the quadratic data pair consists of
$$M=(0, \bigoplus _{i=1}^{r}k\textup{v}_i , \bigoplus _{L\subseteq T} k\mu_L,\dots)$$
and
\begin{align*}R=
\{\,R^1_{i},R^2_{i,j}, R^3_{K,L}, R^4_{i,K}\,|\,  i,j \in T \text{ and } K,L \subseteq T\}\cup
\{\,R^5_{i,K}\,|\, i\in K \subseteq T\}\\
\cup \, \{\,R^6_{i,j,K}\,|\, j\in K \subseteq T\text{ and } i\notin K \}
\end{align*}
where $T=\{1,\ldots,r\}$ and for $i,j\in T$ and $K,L\subseteq T$ with
\\
$R^1_{i}$: \hspace{10pt}
\begin{tikzpicture}[scale=0.9,baseline=-0.5mm]
    \draw [thick](0,0) node[below left] {$t_i$} -- (0,0.3);
    \draw [thick](0,0) node[above left] {$t_i$}-- (0,-0.4);
    \draw[fill] (0,0.3) circle [radius=0.07];
    \draw[fill] (0,-0.4) circle [radius=0.07];
\end{tikzpicture}$\hspace{5pt}=0$\hspace{10pt},
\\
\\
$R^2_{i,j}$: \hspace{10pt}
\begin{tikzpicture}[scale=0.9,baseline=-0.5mm]
    \draw [thick](0,0) node[below left] {$t_j$} -- (0,0.3);
    \draw [thick](0,0) node[above left] {$t_i$}-- (0,-0.4);
    \draw[fill] (0,0.3) circle [radius=0.07];
    \draw[fill] (0,-0.4) circle [radius=0.07];
\end{tikzpicture}$ \hspace{5pt}= $
\begin{tikzpicture}[scale=0.9,baseline=-0.5mm]
    \draw [thick](0,0) node[below left] {$t_i$} -- (0,0.3);
    \draw [thick](0,0) node[above left] {$t_j$}-- (0,-0.4);
    \draw[fill] (0,0.3) circle [radius=0.07];
    \draw[fill] (0,-0.4) circle [radius=0.07];
\end{tikzpicture}\hspace{10pt},
\\
\\
\noindent
$R^3_{K,L}:$
$
\begin{tikzpicture}[scale=0.8,baseline=-0.5mm]
    \draw [thick](0,-0.4) node[below left] {$\mu_L$} -- (-0.6,0.3) node[ left] {$\mu_K$};
    \draw[fill] (-0.6,0.3) circle [radius=0.07];
    \draw[fill] (0,-0.4) circle [radius=0.07];
\end{tikzpicture}
= \left\{
\begin{array}{ll}
    \, \, \,
    \begin{tikzpicture}[scale=0.8,baseline=-0.5mm]
        \draw [thick](0,-0.4) node[below right] {$\mu_K$} -- (0.6,0.3) node[ right] {$\mu_L$};
        \draw[fill] (0.6,0.3) circle [radius=0.07];
        \draw[fill] (0,-0.4) circle [radius=0.07];
    \end{tikzpicture} & \text{ if  $L=\emptyset$ or $K\cap L\neq \emptyset$}\\
    & \\
    \begin{tikzpicture}[scale=0.8,baseline=-0.5mm]
        \draw [thick](0,-0.4) node[below left] {$\mu_0$} -- (-0.6,0.3) node[ left] {$\mu_{K\cup L}$};
        \draw[fill] (-0.6,0.3) circle [radius=0.07];
        \draw[fill] (0,-0.4) circle [radius=0.07];
    \end{tikzpicture} +
    \begin{tikzpicture}[scale=0.8,baseline=-0.5mm]
        \draw [thick](0,-0.4) node[below right] {$\mu_K$} -- (0.6,0.3) node[ right] {$\mu_L$};
        \draw[fill] (0.6,0.3) circle [radius=0.07];
        \draw[fill] (0,-0.4) circle [radius=0.07];
    \end{tikzpicture} +
    \begin{tikzpicture}[scale=0.8,baseline=-0.5mm]
        \draw [thick](0,-0.4) node[below right] {${\mu}_{K\cup L}$} -- (0.6,0.3) node[ right] {$\mu_L$};
        \draw[fill] (0.6,0.3) circle [radius=0.07];
        \draw[fill] (0,-0.4) circle [radius=0.07];
    \end{tikzpicture}  & \text{if $L\neq \emptyset$ and $K\cap L= \emptyset$} \hspace{1pt},\\
\end{array}
\right.
$
\\
\\
\newline
$R^4_{i,K}:$  \hspace{10pt}
$
\begin{tikzpicture}[scale=0.9,baseline=-0.5mm]
    \draw [thick](0,-0.4) node[below left] {$t_i$} -- (0,0.3) node[ left] {$\mu_K$};
    \draw[fill] (0,0.3) circle [radius=0.07];
    \draw[fill] (0,-0.4) circle [radius=0.07];
\end{tikzpicture}
\hspace{5pt}= \left\{
\begin{array}{ll}
    \, \, \,
    \begin{tikzpicture}[scale=0.9,baseline=-0.5mm]
        \draw [thick](0,-0.4) node[below right] {$\mu_K$} -- (0.6,0.3) node[ right] {$t_i$};
        \draw[fill] (0.6,0.3) circle [radius=0.07];
        \draw[fill] (0,-0.4) circle [radius=0.07];
    \end{tikzpicture} & \text{ if  $i \in K$}\\
    & \\
    \begin{tikzpicture}[scale=0.9,baseline=-0.5mm]
        \draw [thick](0,-0.4) node[below left] {$\mu_K$} -- (-0.6,0.3) node[ left] {$t_i$};
        \draw[fill] (-0.6,0.3) circle [radius=0.07];
        \draw[fill] (0,-0.4) circle [radius=0.07];
    \end{tikzpicture} +
    \begin{tikzpicture}[scale=0.9,baseline=-0.5mm]
        \draw [thick](0,-0.4) node[below right] {$\mu_K$} -- (0.6,0.3) node[ right] {$t_i$};
        \draw[fill] (0.6,0.3) circle [radius=0.07];
        \draw[fill] (0,-0.4) circle [radius=0.07];
    \end{tikzpicture} +
    \begin{tikzpicture}[scale=0.9,baseline=-0.5mm]
        \draw [thick](0,-0.4) node[below right] {${\mu_0}_{K\cup \{i\}}$} -- (0.6,0.3) node[ right] {$t_i$};
        \draw[fill] (0.6,0.3) circle [radius=0.07];
        \draw[fill] (0,-0.4) circle [radius=0.07];
    \end{tikzpicture}  & \text{if $i \notin K$ } \hspace{10pt},\\
\end{array}
\right.
$
\\
$R^5_{i,K}:$
\begin{tikzpicture}[scale=0.9,baseline=-0.5mm]
    \draw [thick](0,-0.4) node[below left] {$\mu_K$} -- (-0.6,0.3) node[ left] {$t_i$};
    \draw[fill] (-0.6,0.3) circle [radius=0.07];
    \draw[fill] (0,-0.4) circle [radius=0.07];
\end{tikzpicture}\hspace{5pt}$=0$ \quad if $i \in K$ \hspace{10pt},
\\
$R^6_{i,j,K}:$
\begin{tikzpicture}[scale=0.9,baseline=-0.5mm]
    \draw [thick](0,-0.4) node[below left] {$\mu_K$} -- (-0.6,0.3) node[ left] {$t_i$};
    \draw[fill] (-0.6,0.3) circle [radius=0.07];
    \draw[fill] (0,-0.4) circle [radius=0.07];
\end{tikzpicture}\hspace{5pt}$=$
\begin{tikzpicture}[scale=0.9,baseline=-0.5mm]
    \draw [thick](0,-0.4) node[below right] {$\mu_{\{K\backslash \{j\} \}\cup\{i\}}$} -- (-0.6,0.3) node[ left] {$t_j$};
    \draw[fill] (-0.6,0.3) circle [radius=0.07];
    \draw[fill] (0,-0.4) circle [radius=0.07];
\end{tikzpicture} \quad if $i \notin K$, $j \in K$.
\\
We define a quadratic operad and a quadratic cooperad as follows:
$$\tilde{\mathscr{W}}:=\mathscr{P}(M,R)\text{\ \ \ \ and \ \ \ \ }\tilde{\mathscr{S}}:=\mathscr{C}(sM,s^2R).$$

\subsection{The PBW basis of $\tilde{\mathscr{W}}$}\label{PBWbasisofW}

In order to define a $\mathrm{PBW}$ basis for the operad $\tilde{\mathscr{W}}$, we consider the graded path lexicographic order given in Section \ref{PBWbasis}.
It is straight forward to check that the quadratic part of the basis of $\tilde{\mathscr{W}}$ is given by the following set of trees:
\begin{align*}\left\{ \,\,
\begin{tikzpicture}[scale=0.8,baseline=-0.5mm]
    \draw [thick](0,-0.4) node[below right] {$\mu_K$} -- (0.6,0.3) node[ right] {$t_j$};
    \draw[fill] (0.6,0.3) circle [radius=0.07];
    \draw[fill] (0,-0.4) circle [radius=0.07];
\end{tikzpicture}\,\, \left| \,\,\begin{array}{c}
                    j\in T  \\
                    \text{and}\\
                    K\subseteq T
                  \end{array}\right.\,\,
\right\}
\bigcup
\left\{ \,\,
\begin{tikzpicture}[scale=0.8,baseline=-0.5mm]
    \draw [thick](0,-0.4) node[below right] {$t_j$} -- (0,0.3) node[ right] {$\mu_K$};
    \draw[fill] (0,0.3) circle [radius=0.07];
    \draw[fill] (0,-0.4) circle [radius=0.07];
\end{tikzpicture}\,\,\left| \,\, \begin{array}{c}
                    j\notin K  \\
                    \text{and}\\
                    K\subseteq T
                  \end{array}\right.\,\,
\right\}
\bigcup
\left\{ \,\,
\begin{tikzpicture}[scale=0.8,baseline=-0.5mm]
    \draw [thick](0,-0.4) node[below right] {$\mu_K$} -- (0.6,0.3) node[ right] {$\mu_L$};
    \draw[fill] (0.6,0.3) circle [radius=0.07];
    \draw[fill] (0,-0.4) circle [radius=0.07];
\end{tikzpicture}\,\,\left| \,\,\begin{array}{c}
                                 \\
                        K, L\subseteq T\\
                        \mbox{}
                       \end{array}\right.\,\,
\right\}.
\end{align*}
Note that every composition of the above operations  can be rewritten by a unique basis element.

%\begin{theorem}\cite[Theorem~6.6]{Hoffbeck}
%A non-symmetric operad which has a non-symmetric $\mathrm{PBW}$ basis is Koszul, and the non-symmetric dual operad has a non-symmetric
%$\mathrm{PBW}$ basis, which can be explicitly determined from the other basis.
%\end{theorem}
%For example
%  $\mu_0_{\emptyset}:=\mu_0$ and

\subsection{Multiplicative structure on minimal Hirsch-Brown models}\label{MultiplicativeStr}

Let $G$, $X$, and $C$ be as in Section \ref{minmodelsection}. Let $\Delta: C\rightarrow C\otimes C$ denote the Alexander-Whitney diagonal map. For $g$ in $G$, let $g:C\rightarrow C$ denote the multiplication by $g$ from the left. The map $\Delta$ is coassociative as in \cite{MacLane} and by naturality we have
$$\Delta\circ g= (g\otimes g)\circ \Delta$$
for any $g$ in $G$. Hence for $g$ in $G$, if $t=1+g$ then we have

\begin{equation}\label{equalityDeltaandT}
\begin{split}
\Delta\circ t & =\Delta\circ (1+g)=(\Delta\circ 1)+(\Delta\circ g)=(1\otimes 1)\circ \Delta +  (g\otimes g)\circ \Delta\\
& = ((1\otimes 1) +  (g\otimes g))\circ \Delta = ((t\otimes 1)+(1\otimes t)+(t\otimes t))\circ \Delta .
\end{split}
\end{equation}
Hence, we have a operad morphism defined as follows

\begin{equation} \label{eq1delta}
\begin{split}
\tilde{\mathscr{W}} & \longrightarrow {\mathop{\coEnd}}_C  \\
  t_i&\rightarrow 1+g_i  \\
  \mu_I &\rightarrow \left(\prod_{i\in I}(1+g_i)\ , \ 1\right) \circ \Delta
\end{split}
\end{equation}
for $i$ in $T$ and $I\subseteq T$. By abuse of notation, we will denote the image of $t_i$ and $\mu_I$ under this operad morphism by $t_i$ and $\mu_I$. Hence  $t_i$ and $\mu_I$ will be considered as operations in ${\mathop{\coEnd}}_C$.

In the following lemma, we assume $r=1$, and hence $T=\{1\}$. In this case instead of $g_1, t_1, \mu_{\emptyset}, \mu_T$ we write $g, t, \mu_0,\mu_1$ respectively.
\begin{lemma} \label{Lemmarelationr1}
Assume that $r=1$. Then the operations $t, \mu_0$ and $\mu_1$ in ${\mathop{\coEnd}}_C$ satisfy the relations $R_1, \ldots, R_8$ in Section \ref{ssBVr1}.
\end{lemma}
\begin{proof} We need to prove the following equations:
\begin{align*}
\begin{array}{l}
R_1: \; (t;t)=0,\\
R_2: \ (\mu_1;t,1)=0,\\
R_3:  \; (\mu_0;\mu_0,1)=(\mu_0;1,\mu_0),  \\
R_4: \ (\mu_1;\mu_1,1)=(\mu_1;1,\mu_1),\\
R_5: \; (t; \mu_0)=(\mu_0;t,1)+(\mu_0;1,t))+(\mu_1;1,t), \\
R_6: \; (\mu_0;\mu_1,1)=(\mu_1;1,\mu_0) ,\\
R_7: \ (t;\mu_1)=(\mu_1;1,t), \\
R_8: \; (\mu_1;\mu_0,1)=(\mu_0;\mu_1,1)+(\mu_0;1,\mu_1)+(\mu_1;1,\mu_1).
\end{array}
\end{align*}First notice we have $\mu_1=((1+g),1)\circ \Delta=(\mu_0;t,1)$ in ${\mathop{\coEnd}}_C$. We will call this equation $R_0$.

The first equation $(t;t)=0$ holds since $t=1+g, g^2=1$ and $(1+g)^2=1+2g+g^2=0$ in a characteristic $2$ field. The second equation $(\mu_1;t,1)=((\mu_0;t,1);t,1)=(\mu_0;t^2,1)=0$ follows from $R_0$ and $R_1$. By associativity of the operation $\mu_0$, we have the equation $R_3$.
The equation $R_5$ is given by the Equation \ref{equalityDeltaandT}.

The equation $R_4$ follows from $R_0$, $R_5$, $R_1$ and $R_3$. More precisely,
 one can obtain $(\mu_1;\mu_1,1)=(\mu_0;(t;\mu_1),1)=(\mu_0;(t;(\mu_0;t,1)),1)$ by $R_0$. Then
\begin{align*}
(\mu_1;\mu_1,1)& =(\mu_0;(\mu_0;t^2,1),1)+(\mu_0;(\mu_0;t,t),1))+(\mu_0;(\mu_0;t^2,t),1)\text{ by $R_5$}, \\
                         & =(\mu_0;(\mu_0;t,t),1)) \text{ by  $R_1$},\\
                         & =((\mu_0;\mu_0,1);t,t,1),\\
                         &= ((\mu_0;1,\mu_0);t,t,1) \text{ by  $R_3$}, \\
                         &=(\mu_0;t,(\mu_0;t,1))=(\mu_1;1,\mu_1).
\end{align*}
%%Hence $\mu_1$ is associative.

The equation $R_6$ can be seen as
\begin{align*}
(\mu_0;\mu_1,1)=(\mu_0;(\mu_0;t,1),1)&=((\mu_0;\mu_0,1);t,1,1)\text { by $R_0$;}\\
&=((\mu_0;1,\mu_0);t,1,1) \text{ by $R_3$},\\
&=(\mu_0;t,(\mu_0;1,1)),\\
&=((\mu_0;t,1);1,\mu_0)=(\mu_1;1,\mu_0).
\end{align*}
The equation $R_7$ follows from $R_0$, $R_5$ and $R_1$:
\begin{align*}
(t;\mu_1)=(t;(\mu_0;t,1))=(\mu_0;t^2,1)+(\mu_0;t,t)+(\mu_0;t^2,t)=(\mu_0;t,t)=(\mu_1;1,t).
\end{align*}
The last equation $R_8$ follows from $R_0$, $R_5$, $R_7$, $R_3$ and $R_4$. More precisely,
\begin{align*}
(\mu_1;\mu_0,1)&=((\mu_0;t,1);\mu_0,1)\text { by }R_0,\\
&=(\mu_0;(t;\mu_0),1),\\
&=(\mu_0;(\mu_0;t,1),1)+(\mu_0;(\mu_0;1,t),1)+(\mu_0;(\mu_1;1,t),1)\text { by }R_5,\\
&=(\mu_0;(\mu_0;t,1),1)+((\mu_0;\mu_0,1);1,t,1)+(\mu_0;(t;(\mu_0;t,1)),1)\text { by }R_7,\\
&=(\mu_0;(\mu_0;t,1),1)+((\mu_0;1,\mu_0);1,t,1)+(\mu_0;(t;(\mu_0;t,1)),1)\text { by }R_3,\\
&=(\mu_0;\mu_1,1)+(\mu_0;1,\mu_1)+(\mu_1;\mu_1,1) \text { by }R_0,\\
&=(\mu_0;\mu_1,1)+(\mu_0;1,\mu_1)+(\mu_1;1, \mu_1) \text { by }R_4.
\end{align*}

\end{proof}
\begin{lemma}\label {Lemmarelationr2}
The operations $t_1, t_2,\ldots, t_r$ and $\mu_L$ for $L\subseteq T=\{1,\ldots,r\}$ in ${\mathop{\coEnd}}_C$ satisfy the equations $R^1_i, R^2_{i,j}, R^3_{K,L}, R^4_{i,K}, R^5_{i,K}$ and $R^6_{i,j,K}$ listed in Section \ref{ssBV}.
\end{lemma}
\begin{proof}
In the operad ${\mathop{\coEnd}}_C$, we have $\displaystyle \mu_L=\Big(\mu_0;\prod_{i\in L} t_i,1\Big)$. Hence Lemma \ref{Lemmarelationr1} can be repeatedly applied for $t_i$ at a time to prove this lemma.
\end{proof}

Note that by Equation \ref{eq1delta}, we can consider $C$ as a $\dg$-$\tilde{\mathscr{W}}$-coalgebra and hence $C^*$ as a $\dg$-$\tilde{\mathscr{W}}$-algebra. We will call the $\tilde{\mathscr{S}}$-coalgebra $\barc_{\iota} H(C^*)$ the minimal
Hirsch-Brown model as at the end of Section \ref{minmodelsection}. Now, note that the inclusion ${\As}\rightarrowtail \tilde{\mathscr{W}}$,
induces
$\As^{{\mbox{!`} }}\rightarrowtail \tilde{\mathscr{S}}$, and so it induces ${\tilde{\mathscr{S}}}^*\twoheadrightarrow  (\As^{{\mbox{!`} }})^*=\As $. We also have surjective morphism $\tilde{\mathscr{W}}\rightarrowtail \As $ obtained by sending $t_i$ to $0$ for $i$ in $T$ and $\mu_L$ to $0$ for $\emptyset\neq L\subseteq T$. This induces $\tilde{\mathscr{S}}\twoheadrightarrow \As^{{\mbox{!`} }}$, and so it induces $\As= (\As^{{\mbox{!`} }})^*\rightarrowtail{\tilde{\mathscr{S}}}^* $. Notice that the composition
$$\As\rightarrowtail{\tilde{\mathscr{S}}}^*\twoheadrightarrow  \As $$
is the identity morphism on $\As$. So this means we have a multiplicative structure on duals of these minimal Hirsh-Brown models. Unfortunately, this multiplicative structure does not have all the properties of the multiplicative structure used by Puppe, therefore we cannot repeat the proof of \cite[Lemma~2.1.a]{Puppe} to obtain an equivalent result. We hope that results stronger than \cite[Lemma~2.1.b]{Puppe} can be proved to tighten the bounds mentioned in Section \ref{intro}. The reason for this hope is because we also know that the composition
$$\mathscr{S}^*\rightarrowtail{\tilde{\mathscr{S}}}^*\twoheadrightarrow \mathscr{S}^*  $$
is the identity morphism on  $\mathscr{S}^*$, where the first morphism is induced by sending $\mu_L$ to $0$ for $L\subseteq T$ and the second morphism comes from the inclusion $\As\rightarrowtail \tilde{\mathscr{W}}$. Hence we have an $S$-module structure, which is enough to prove \cite[Lemma~2.1.b]{Puppe}.

\begin{proof}[Proof of Theorem \ref{mainthm}] Take $\mathscr{P}=\tilde{\mathscr{W}}$ where $\tilde{\mathscr{W}}$ is the operad constructed in Section \ref{ssBV}. Hence $\mathscr{P}$ is a Kozul operad by Theorem \ref{HofbeckThm66} and Section \ref{PBWbasisofW}. Notice that the Koszul dual operad of $\mathscr{P}$ is $\tilde{\mathscr{S}}^*$. Hence the other properties of $\mathscr{P}$ are proved in Section \ref{MultiplicativeStr}.
\end{proof}

\begin{proof}[Proof of Proposition \ref{mainprop}] Let $X$ be a finite-dimensional free $G$-simplicial set. Notice that the group $G$ acts freely on products of $r$ many equidimensional spheres where $g_i$ acts on the $i$th sphere by the antipodal action. Hence $G$ acts freely on $\mathbb{S}^0 \times \ldots \times \mathbb{S}^0 $ and $\mathbb{S}^{m} \times \ldots \times \mathbb{S}^{m}$, where $m=\dim (X)+1$. Let $x_0$ be a point in $X$. First, we can define an equivariant map from $\mathbb{S}^0 \times \ldots \times \mathbb{S}^0 $ to $X$ by sending the south poles $s:=(-1,\ldots, -1)$ to $x_0$ and extending the map equivariantly. Second, we can construct a map from $X$ to $\mathbb{S}^{m} \times \ldots \times \mathbb{S}^{m}$ by sending $x_0$ to $i_m(s)$ then extending the map by using equivariant obstruction theory; see \cite[ Chapter~II]{BredonEquivariantCohomologyTheories}. Since the associated minimal models of products of equidimensional spheres are Koszul complexes, we obtain the result of the proposition by naturality of our constructions.
\end{proof}
\begin{proof}[Proof of Proposition \ref{mainprop2}]
Let $\mathscr{P}$ denote the operad in Theorem \ref{mainthm}, $m$ be a positive integer and $\gamma:{K_r}(m)\rightarrow {K_r}(0)$  be a $S^*$-coalgebra morphism which induces the same map from $\Homology (K_r(m))$ to $\Homology( K_r(0))$ as $i_m^*$ does. Assume $\gamma$  extends to a $\mathscr{P}^{\mbox{!`}}$-coalgebra morphism $\tilde{\gamma}:\tilde{K_r}(m)\longrightarrow \tilde{K_r}(0)$ whose restriction to $\mathscr{P}^{\mbox{!`}}\circ 1$ is induced by the identity on $\mathscr{P}^{\mbox{!`}}$.
Set $T=\{1,\ldots,r\}$. For $U\subseteq T$ and $n\in \{0,m\}$, let $z^{(n)}_U$ denote $\displaystyle\prod_{i\in U}z_i^{(n)}$ and $x_U$ denote $\displaystyle\prod_{i\in U}x_i^{m}$. Since $\gamma $ induces the same map from $\Homology (K_r(m))$ to $\Homology( K_r(0))$ as $i_m^*$ does,
 $\gamma ^*$ sends $[1\otimes z^{(0)}_{T}]$ to $[x_T\otimes z^{(m)}_{T}]$. Since $\gamma$ extends to $\mathscr{P}^{\mbox{!`}}$-coalgebra map whose restriction to $\mathscr{P}^{\mbox{!`}}\circ 1$ is induced by the identity on $\mathscr{P}^{\mbox{!`}}$, we can say that
$\gamma ^*$ sends $x_U\otimes z^{(0)}_{T-U}$ to $x_T\otimes z^{(m)}_{T-U} + e$  where $e$ contains only terms in the form $f\otimes z_L$ such that
$T-U$ is a proper subset of $L$ and $L$ is a subset of $T$. Hence $F\otimes \gamma^*$ sends
$z_{\emptyset}, z_{\{1\}},\ldots ,z_{\{r\}}, z_{\{1,2\}},\ldots ,z_{\{1,r\}}$ to linearly independent vectors in $F\otimes _S ({K_r}(m))^*$. Then the linear map $F\otimes _S \gamma^*$ has rank at least $2r$, where $F$ denotes the field of fractions of the ring $S$.
\end{proof}
\noindent \textbf{Acknowledgments.}
We would like to thank Marc Stephan for helpful discussions.

\end{document}